\documentclass[a4paper,reqno,11pt,dvipdfmx]{amsart}
\usepackage[tmargin=30mm,bmargin=30mm,hmargin=25mm]{geometry}
\usepackage{amsmath}
\usepackage{amssymb}
\usepackage{amsthm}
\usepackage{enumerate}
\usepackage{mathrsfs}
\input xy
\xyoption{all}
\numberwithin{equation}{section}
\usepackage{hyperref}

\newtheorem{theorem}{Theorem}[section]
\newtheorem{corollary}[theorem]{Corollary}
\newtheorem{lemma}[theorem]{Lemma}
\newtheorem{proposition}[theorem]{Proposition}
\theoremstyle{definition}
\newtheorem{definition}[theorem]{Definition}
\newtheorem{remark}[theorem]{Remark}
\newtheorem{example}[theorem]{Example}

\begin{document}
\title[Silting subcategories of extriangulated categories]{An assortment of properties of silting subcategories of extriangulated categories}

\author{Takahide Adachi}
\address{T.~Adachi: Faculty of Global and Science Studies, Yamaguchi University, 1677-1 Yoshida, Yamaguchi 753-8541, Japan}
\email{tadachi@yamaguchi-u.ac.jp}

\author{Mayu Tsukamoto}
\address{M.~Tsukamoto: Graduate school of Sciences and Technology for Innovation, Yamaguchi University, 1677-1 Yoshida, Yamaguchi 753-8512, Japan}
\email{tsukamot@yamaguchi-u.ac.jp}


\begin{abstract}

Extriangulated categories give a simultaneous generalization of triangulated categories and exact categories.
In this paper, we study silting subcategories of an extriangulated category.
First, we show that a silting subcategory induces a basis of the Grothendieck group of an extriangulated category.
Secondly, we introduce the notion of silting mutation and investigate its basic properties.
Thirdly, we explore properties of silting subcategories of the subcategory consisting of objects with finite projective dimension.
As an application, we can recover Auslander--Reiten's result which gives a bijection between tilting modules and contravariantly finite resolving subcategories with finite projective dimension.

\end{abstract}
\maketitle

\section{Introduction}

The concept of silting objects was firstly introduced by Keller and Vossieck (\cite{KV88}) to study $t$-structures on the bounded derived category of representations of a Dynkin quiver. 
This is a generalization of tilting complexes and tilting modules.
Nowadays, silting objects play a crucial role in the study of triangulated categories. 
It is known that silting objects are closely related to several important objects in triangulated categories such as $t$-structures, co-$t$-structures, simple-minded collections, and so on (e.g., \cite{AMY19, B10, BY13, KY14, MSSS13, SY19}). 
As an analog of silting objects in a triangulated category, the authors (\cite{AT22}) introduced the notion of silting objects in an extriangulated category defined by Nakaoka and Palu (\cite{NP19}) as a simultaneous generalization of triangulated categories and exact categories.
Tilting modules are realized as silting objects in the subcategory of a module category consisting of modules with finite projective dimension, which forms an extriangulated category.
Namely, silting objects in an extriangulated category are a common generalization of silting objects in a triangulated category and tilting modules over a finite dimensional algebra. 

In this paper, we study properties of silting objects/subcategories in an extriangulated category. 
First, we describe the Grothendieck groups of an extriangulated categories with a silting subcategory.
The Grothendieck group is known as a basic invariant for a triangulated category and an exact category. 
Aihara and Iyama (\cite{AI12}) showed that non-isomorphic indecomposable objects in a silting subcategory form a basis of the Grothendieck group of a Krull--Schmidt  triangulated category.
Recently, the Grothendieck group of an extriangulated category is defined in \cite{H21, ZZ21}. 
As an extriangulated version of \cite[Theorem 2.27]{AI12}, we obtain the following result.

\begin{theorem}[{Theorem \ref{mainresult-Grgroup}, Corollary \ref{cor-Grgroup}}]\label{intro-thm1}
Let $\mathcal{C}$ be a Krull--Schmidt extriangulated category.
If $\mathcal{C}$ admits a silting subcategory $\mathcal{M}$, then the Grothendieck group $K_{0}(\mathcal{C})$ of $\mathcal{C}$ is a free abelian group and its basis consists of the set of isomorphism classes of indecomposable objects in $\mathcal{M}$. 
In particular, all silting objects have the same number of non-isomorphic indecomposable direct summands. 
\end{theorem}

Secondly, we introduce silting mutation in an extriangulated category as a common generalization of tilting mutation in the module category over a finite dimensional algebra introduced by Riedtmann and Schofield (\cite{RS91}), and silting mutation in a triangulated category introduced by Aihara and Iyama (\cite{AI12}). 
Let $\mathcal{C}$ be an extriangulated category and let $\mathcal{M}$ be a silting subcategory of $\mathcal{C}$.
When $\mathcal{D}$ is a good covariantly finite subcategory of $\mathcal{M}$, that is, each object in $\mathcal{M}$ admits a left $\mathcal{D}$-approximation which is an $\mathfrak{s}$-inflation (in $\mathcal{C}$), we define a left mutation $\mu^{L}(\mathcal{M};\mathcal{D})$ of $\mathcal{M}$ with respect to $\mathcal{D}$ (see Definition \ref{def-mut}).
The following theorem is one of main results of this paper. 
  
\begin{theorem}[{Theorem \ref{thm-main2-1845}(3)}]\label{intro-thm2}
Let $\mathcal{M}$ be a silting subcategory of $\mathcal{C}$ and let $\mathcal{D}$ be a good covariantly finite subcategory of $\mathcal{M}$. 
Then $\mu^{L}(\mathcal{M};\mathcal{D})$ is a silting subcategory of $\mathcal{C}$. 
\end{theorem}

Happel and Unger (\cite{HU05}) and Aihara and Iyama (\cite{AI12}) proved that tilting mutation and silting mutation are closely related to a partial order of tilting modules and silting subcategories respectively. 
As a counterpart of their results, we give a relationship between silting mutation and a partial order of silting subcategories in an extriangulated category (see Theorem \ref{thm-comb}).  

Using silting mutation, we study a Bongartz-type lemma for silting subcategories in an extriangulated category. 
Bongartz (\cite{B81}) proved that each pretilting module with projective dimension at most one is a direct summand of a tilting module with projective dimension at most one.
However, it is known that a Bongartz-type lemma for tilting modules does not necessarily holds.
Rickard and Schofield (\cite{RS89}) showed that a Bongartz-type lemma for tilting modules with finite projective dimension holds true for algebras of finite representation type. 
On the other hand, Aihara and Mizuno (\cite{A13,AM17}) investigated a Bongartz-type lemma for silting objects in a triangulated category and gave its sufficient conditions.
As an analog of their results, we provide some sufficient conditions for a Bongartz-type lemma to hold true (see Theorem \ref{thm-BC1}, Corollary \ref{cor-BC1} and Theorem \ref{thm-315}). 

Finally, we give a relationship between silting subcategories and contravariantly finite resolving subcategories, which generalizes Auslander--Reiten's result (\cite[Theorem 5.5]{AR91}). 
Let $\mathcal{P}^{\infty}$ denote the subcategory of $\mathcal{C}$ consisting of objects with finite projective dimension. 
Then we have the following result.
 
\begin{theorem}[{Theorem \ref{thm-AR2}}]\label{intro-thm3}
Let $\mathcal{C}$ be a weakly idempotent complete extriangulated category which has enough projective objects and enough injective objects.  
Assume that each silting subcategory of $\mathcal{P}^{\infty}$ with finite projective dimension is contravariantly finite. 
Then there exist mutually inverse bijections between the set of silting subcategories of $\mathcal{P}^{\infty}$ with finite projective dimension and the set of contravariantly finite resolving subcategories of $\mathcal{C}$ with finite projective dimension. 
\end{theorem}

By Theorem \ref{intro-thm3}, we obtain a bijection between the set of isomorphism classes of basic tilting modules and the set of contravariantly finite resolving subcategories consisting of modules with finite projective dimension for a finite dimensional algebra (see Corollary \ref{cor_AR}).

\section{Extriangulated categories}

Throughout this paper, $R$ is a commutative unital ring, $\mathcal{C}$ is a small additive $R$-linear category and all subcategories are assumed to be strictly full.

In this section, we collect terminologies and basic properties of extriangulated categories which we need later. 
We omit the precise definition of extriangulated categories.
For details, we refer to \cite{INP, NP19}.
An extriangulated category $\mathcal{C}=(\mathcal{C},\mathbb{E},\mathfrak{s})$ consists of the following data which satisfy certain axioms (see \cite[Definition 2.12]{NP19}):
\begin{itemize}
\item $\mathcal{C}$ is a small additive $R$-linear category.
\item $\mathbb{E}:\mathcal{C}^{\mathrm{op}}\times \mathcal{C}\to \operatorname{\mathsf{Mod}}R$ is an $R$-bilinear functor.
\item $\mathfrak{s}$ is a correspondence which associates an equivalence class $[A\rightarrow B\rightarrow C]$ of complexes in $\mathcal{C}$ to each $\delta\in \mathbb{E}(C,A)$.
Here two complexes $A\xrightarrow{f}B\xrightarrow{g}C$ and $A\xrightarrow{f'}B'\xrightarrow{g'}C$ in $\mathcal{C}$ are \emph{equivalent} if there exists an isomorphism $b:B\to B'$ such that the diagram
\begin{align}
\xymatrix{
A\ar[r]^-{f}\ar@{=}[d]&B\ar[r]^-{g}\ar[d]_{b}^{\cong}&C\ar@{=}[d]\\
A\ar[r]^-{f'}&B'\ar[r]^-{g'}&C
}\notag
\end{align}
is commutative, and let $[A\xrightarrow{f}B\xrightarrow{g}C]$ denote the equivalence class of $A\xrightarrow{f}B\xrightarrow{g}C$.
\end{itemize}
A complex $A\xrightarrow{f}B\xrightarrow{g}C$ in $\mathcal{C}$ is called an \emph{$\mathfrak{s}$-conflation} if there exists $\delta\in \mathbb{E}(C,A)$ such that $\mathfrak{s}(\delta)=[A\xrightarrow{f}B\xrightarrow{g}C]$.
We write the $\mathfrak{s}$-conflation as $A\xrightarrow{f}B\xrightarrow{g}C\overset{\delta}{\dashrightarrow}$.
Then $f$ is called an \emph{$\mathfrak{s}$-inflation}, $g$ is called an \emph{$\mathfrak{s}$-deflation}, $C$ is denoted by $\mathrm{Cone}(f)$ and $A$ is denoted by $\mathrm{Cocone}(g)$. 

In the rest of this paper, the following subcategories play a crucial role. 

\begin{definition}
Let $\mathcal{X}, \mathcal{Y}$ be subcategories of $\mathcal{C}$. 
\begin{enumerate}[\upshape(1)]
\item Let $\mathcal{X} \ast \mathcal{Y}$ denote the subcategory of $\mathcal{C}$ consisting of $M\in\mathcal{C}$ which admits an $\mathfrak{s}$-conflation $X \rightarrow M \rightarrow Y \dashrightarrow$ in $\mathcal{C}$ with $X\in \mathcal{X}$ and $Y\in\mathcal{Y}$. We say that \emph{$\mathcal{X}$ is closed under extensions} if $\mathcal{X}\ast\mathcal{X}\subseteq \mathcal{X}$.
\item Let $\operatorname{\mathsf{cone}}(\mathcal{X},\mathcal{Y})$ denote the subcategory of $\mathcal{C}$ consisting of $M\in\mathcal{C}$ which admits an $\mathfrak{s}$-conflation $X \rightarrow Y \rightarrow M \dashrightarrow$ in $\mathcal{C}$ with $X\in \mathcal{X}$ and $Y\in\mathcal{Y}$.  We say that \emph{$\mathcal{X}$ is closed under cones} if $\operatorname{\mathsf{cone}}(\mathcal{X},\mathcal{X})\subseteq\mathcal{X}$. 
\item For each $n\geq 0$, we inductively define a subcategory $\mathcal{X}^{\wedge}_{n}$ of $\mathcal{C}$ as $\mathcal{X}^{\wedge}_{n}:=\operatorname{\mathsf{cone}}(\mathcal{X}^{\wedge}_{n-1},\mathcal{X})$, where $\mathcal{X}^{\wedge}_{-1}:=\{ 0 \}$.
Put $\displaystyle \mathcal{X}^{\wedge}:=\bigcup_{n\geq 0}\mathcal{X}^{\wedge}_{n}$.
\item Let $\operatorname{\mathsf{cocone}}(\mathcal{X}, \mathcal{Y})$ denote the subcategory of $\mathcal{C}$ consisting of $M\in\mathcal{C}$ which admits an $\mathfrak{s}$-conflation $M\rightarrow X \rightarrow Y\dashrightarrow$ in $\mathcal{C}$ with $X\in \mathcal{X}$ and $Y\in\mathcal{Y}$. We say that \emph{$\mathcal{X}$ is closed under cocones} if $\operatorname{\mathsf{cocone}}(\mathcal{X},\mathcal{X})\subseteq\mathcal{X}$.
\item For each $n\geq 0$, we inductively define a subcategory $\mathcal{X}^{\vee}_{n}$ of $\mathcal{C}$ as $\mathcal{X}^{\vee}_{n}:=\operatorname{\mathsf{cocone}}(\mathcal{X},\mathcal{X}^{\vee}_{n-1})$, where $\mathcal{X}^{\vee}_{-1}:=\{ 0 \}$.
Put $\displaystyle \mathcal{X}^{\vee}:=\bigcup_{n\geq 0}\mathcal{X}^{\vee}_{n}$. 
\item We call $\mathcal{X}$ a \emph{thick subcategory} of $\mathcal{C}$ if it is closed under extensions, cones, cocones and direct summands. Let $\operatorname{\mathsf{thick}}\mathcal{X}$ denote the smallest thick subcategory containing $\mathcal{X}$. 
\item We call $\mathcal{X}$ a \emph{resolving subcategory} of $\mathcal{C}$ if $\mathcal{C}=\operatorname{\mathsf{cone}}(\mathcal{C},\mathcal{X})$ and it is closed under extensions, cocones and direct summands.
\end{enumerate}
\end{definition}

An object $P\in\mathcal{C}$ is said to be \emph{projective} if $\mathbb{E}(P,\mathcal{C})=0$. 
Let $\operatorname{\mathsf{proj}}\mathcal{C}$ denote the subcategory of $\mathcal{C}$ consisting of all projective objects in $\mathcal{C}$.
Dually, we define injective objects and the subcategory $\operatorname{\mathsf{inj}}\mathcal{C}$.
We can check that if $\mathcal{C}$ has enough projective objects (i.e., $\mathcal{C}=\operatorname{\mathsf{cone}}(\mathcal{C},\operatorname{\mathsf{proj}}\mathcal{C})$), then $\mathcal{X}$ is a resolving subcategory if and only if $\operatorname{\mathsf{proj}}\mathcal{C}\subseteq\mathcal{X}$ and it is closed under extensions, cocones and direct summands. 

Gorsky, Nakaoka and Palu (\cite{GNP}) gave an $R$-bilinear functor $\mathbb{E}^{n}: \mathcal{C}^{\mathrm{op}} \times \mathcal{C} \to \operatorname{\mathsf{Mod}}R$ and proved that any $\mathfrak{s}$-conflation induces the following long exact sequences.

\begin{proposition}[{\cite[Theorem 3.5]{GNP}}]\label{prop_longex}
Let $A\rightarrow B\rightarrow C\dashrightarrow$ be an $\mathfrak{s}$-conflation. 
Then the following statements hold. 
\begin{enumerate}[\upshape(1)]
\item For each $X\in\mathcal{C}$, there exists a long exact sequence 
\begin{align}
&\mathcal{C}(X,A) \to \mathcal{C}(X,B) \to \mathcal{C}(X,C) \to \mathbb{E}(X,A) \to \cdots \notag  \\ 
&\cdots \to \mathbb{E}^{n-1}(X,C) \to \mathbb{E}^{n}(X,A) \to \mathbb{E}^{n}(X,B) \to \mathbb{E}^{n}(X,C) \to \cdots. \notag
\end{align}
\item For each $X\in \mathcal{C}$, there exists a long exact sequence 
\begin{align}
&\mathcal{C}(C, X) \to \mathcal{C}(B,X) \to \mathcal{C}(A,X) \to \mathbb{E}(C,X) \to \cdots \notag \\ 
&\cdots \to \mathbb{E}^{n-1}(A,X) \to \mathbb{E}^{n}(C,X) \to \mathbb{E}^{n}(B,X) \to \mathbb{E}^{n}(A,X) \to \cdots. \notag
\end{align}
\end{enumerate}
\end{proposition}

Recall typical examples of extriangulated categories (for detail, see \cite{GNP, NP19}).

\begin{example}\label{example-etcat}
\begin{enumerate}[\upshape(1)]
\item Let $\Lambda$ be a finite dimensional algebra and let $\operatorname{\mathsf{mod}}\Lambda$ denote the category of finitely generated right $\Lambda$-modules.
Then $\operatorname{\mathsf{mod}}\Lambda$ becomes an extriangulated category by the following data.
\begin{itemize}
\item[$\bullet$] $\mathbb{E}(C,A):=\mathrm{Ext}_{\Lambda}^{1}(C,A)$ for all $A,C \in \operatorname{\mathsf{mod}}\Lambda$.
\item[$\bullet$] $\mathfrak{s}$ is the identity. 
\end{itemize} 
In this case, we have $\mathbb{E}^{k}(C, A)=\mathrm{Ext}^{k}_{\Lambda}(C,A)$ for all $A, C \in \operatorname{\mathsf{mod}}\Lambda$ and $k \ge 1$. 
\item Let $\mathcal{D}$ be a triangulated category with shift functor $\Sigma$. 
Then $\mathcal{D}$ becomes an extriangulated category by the following data.
\begin{itemize}
\item[$\bullet$] $\mathbb{E}(C, A):=\mathcal{D}(C,\Sigma A)$ for all $A, C \in \mathcal{D}$. 
\item[$\bullet$] For $\delta\in \mathbb{E}(C,A)$, we take a triangle $A\xrightarrow{f}B\xrightarrow{g}C\xrightarrow{\delta}\Sigma A$.
Then we define $\mathfrak{s}(\delta):=[A\xrightarrow{f}B\xrightarrow{g}C]$.
\end{itemize} 
In this case, we have $\mathbb{E}^{k}(C, A)=\mathcal{D}(C,\Sigma^{k} A)$ for all $A, C \in \mathcal{D}$ and $k \ge 1$. 
\item Let $\mathcal{C}$ be an extriangulated category and let $\mathcal{X}$ be a subcategory of $\mathcal{C}$ which is closed under extensions.  
Then by restricting the extriangulated structure to $\mathcal{X}$, we can regard $\mathcal{X}$ as an extriangulated category (see \cite[Remark 2.18]{NP19}). 
\end{enumerate}
\end{example}

We give two remarks on the $R$-bilinear functor $\mathbb{E}^{n}$.
If $\mathcal{C}$ has enough projective objects and enough injective objects, then the bilinear functor $\mathbb{E}^{n}$ is isomorphic to that in \cite{HLN21} or \cite{LN19} (see \cite[Corollary 3.21]{GNP}).
Let $\mathcal{X}$ be a subcategory of $\mathcal{C}$ which is closed under extensions.  
For $M,N\in\mathcal{X}$, let $\mathbb{E}^{i}_{\mathcal{X}}(M,N)$ denote the $i$-th positive extension group of the extriangulated category $\mathcal{X}$. 
Although we have $\mathbb{E}_{\mathcal{X}}(M,N)=\mathbb{E}(M,N)$ for each $M,N\in\mathcal{X}$, it does not necessarily satisfy $\mathbb{E}^{i}_{\mathcal{X}}(M,N)=\mathbb{E}^{i}(M,N)$ for $i \geq 2$ (see \cite[Example 3.30]{GNP}). 

If $\mathcal{X}$ is resolving, then the positive extension group of $\mathcal{X}$ is isomorphic to that of $\mathcal{C}$. 

\begin{lemma}\label{lem-thickE}
Let $\mathcal{C}$ be an extriangulated category with enough projective objects and let $\mathcal{X}$ be a resolving subcategory of $\mathcal{C}$.  
Then we have an isomorphism $\mathbb{E}_{\mathcal{X}}^{i}(M,N)\cong \mathbb{E}^{i}(M,N)$ for each $M,N\in\mathcal{X}$ and $i \geq 1$. 
\end{lemma}

\begin{proof}
Let $M,N\in\mathcal{X}$. 
Since $\mathcal{C}$ has enough projective objects, we have an $\mathfrak{s}$-conflation $K_{1}\rightarrow P_{0}\rightarrow M\dashrightarrow$ with $P_{0}\in\operatorname{\mathsf{proj}}\mathcal{C}$. 
Since $\mathcal{X}$ is closed under cocones and contains projective objects, we have $K_{1}\in\mathcal{X}$. 
Inductively, there exists an $\mathfrak{s}$-conflation $K_{i+1}\rightarrow P_{i}\rightarrow K_{i}\dashrightarrow$ such that $K_{i+1}, K_{i}\in\mathcal{X}$ and $P_{i}\in\operatorname{\mathsf{proj}}\mathcal{C}$. 
Applying $\mathcal{X}(-,N)$ to the $\mathfrak{s}$-conflation gives an isomorphism $\mathbb{E}_{\mathcal{X}}^{i}(M,N)\cong \mathbb{E}_{\mathcal{X}}(K_{i-1},N)=\mathbb{E}(K_{i-1},N)$ for each $i\geq 1$. 
On the other hand, applying $\mathcal{C}(-,N)$ to the $\mathfrak{s}$-conflation gives an isomorphism $\mathbb{E}^{i}(M,N)\cong \mathbb{E}(K_{i-1},N)$ for each $i\geq 1$. 
Thus we have the assertion. 
\end{proof}

\section{Silting subcategories and hereditary cotorsion pairs}

Let $\mathcal{C}$ be an extriangulated category.
In this section, we collect basic properties of silting subcategories and hereditary cotorsion pairs.

\subsection{Definition and basic properties of silting subcategories}
In this subsection, we recall the definition of silting subcategories and collect their properties. For detail, see \cite{AT22,LZZZ}.
For a class $\mathcal{X}$ of objects in $\mathcal{C}$, let $\operatorname{\mathsf{add}}\mathcal{X}$ denote the subcategory of $\mathcal{C}$ whose objects are direct summands of finite direct sums of objects in $\mathcal{X}$.

\begin{definition}
Let $\mathcal{M}$ be a subcategory of $\mathcal{C}$.
\begin{enumerate}[(1)]
\item $\mathcal{M}$ is said to be \emph{presilting} if $\mathcal{M}=\operatorname{\mathsf{add}}\mathcal{M}$ and $\mathbb{E}^{k}(\mathcal{M},\mathcal{M})=0$ for all $k\geq 1$.
\item A presilting subcategory $\mathcal{M}$ is called a \emph{silting subcategory} of $\mathcal{C}$ if $\mathcal{C}=\operatorname{\mathsf{thick}} \mathcal{M}$.
Let $\operatorname{\mathsf{silt}}\mathcal{C}$ denote the set of all silting subcategories of $\mathcal{C}$. 
\item A presilting subcategory $\mathcal{M}$ is called a \emph{partial silting subcategory} of $\mathcal{C}$ if there exists a silting subcategory $\mathcal{N}$ such that $\mathcal{M}\subseteq \mathcal{N}$.
\end{enumerate}
\end{definition}

An object $M\in\mathcal{C}$ is said to be \emph{presilting} (respectively, \emph{silting}, \emph{partial silting}) if $\operatorname{\mathsf{add}}M$ is presilting (respectively, silting, partial silting). 
If $\mathcal{C}$ is a Krull--Schmidt category, then $M\mapsto \operatorname{\mathsf{add}}M$ gives a one-to-one correspondence between the set of isomorphism classes of basic objects and the set of subcategories of $\mathcal{C}$ containing additive generators. 
If $\mathcal{C}$ admits a silting object, then each silting subcategory contains an additive generator (for example, see \cite[Proposition 5.4]{AT22}). 
Thus we can identify $\operatorname{\mathsf{silt}}\mathcal{C}$ with the set of isomorphism classes of basic silting objects.

For an integer $n\geq 0$ and a subcategory $\mathcal{X}$ of $\mathcal{C}$, we define a subcategory $\mathcal{X}^{\perp_{>n}}$ as 
\begin{align}
\mathcal{X}^{\perp_{>n}}:=\{M\in\mathcal{C}\mid \mathbb{E}^{k}(\mathcal{X},M)=0 \textnormal{\; for each } k \geq n+1\}.\notag
\end{align}
Let $\mathcal{X}^{\perp}:=\mathcal{X}^{\perp_{>0}}$.
Dually, we define subcategories ${}^{\perp_{>n}}\mathcal{X}$ and ${}^{\perp}\mathcal{X}$. 
For a presilting subcategory $\mathcal{M}$,  we collect properties of $\mathcal{M}^{\wedge}$ and $\mathcal{M}^{\vee}$ (see \cite[Proposition 4.7, Proposition 4.8, Lemma 4.9, Proposition 4.10 and Lemma 4.11]{AT22}).

\begin{lemma}\label{lem-psilt_wedge}
Let $\mathcal{M}$ be a presilting subcategory of $\mathcal{C}$ and let $n,m$ be non-negative integers. 
Then the  following statements hold. 
\begin{enumerate}[\upshape(1)]
\item $\mathcal{M}^{\wedge}_{n}$ and $\mathcal{M}^{\vee}_{n}$ are closed under extensions and direct summands.  
\item We have $(\mathcal{M}_{n}^{\wedge})_{m}^{\vee}=\operatorname{\mathsf{cocone}}(\mathcal{M}_{n}^{\wedge},\mathcal{M}_{m-1}^{\vee})=\operatorname{\mathsf{cone}}(\mathcal{M}_{n-1}^{\wedge},\mathcal{M}_{m}^{\vee})=(\mathcal{M}^{\vee}_{m})_{n}^{\wedge}$.  
\item We have $(\mathcal{M}^{\wedge})^{\vee}=(\mathcal{M}^{\vee})^{\wedge}=\operatorname{\mathsf{thick}}\mathcal{M}$.
\item $\mathcal{M}^{\wedge}=\operatorname{\mathsf{thick}}\mathcal{M}\cap\mathcal{M}^{\perp}$ holds. In particular, $\mathcal{M}^{\wedge}$ is closed under cones.
If $\mathcal{M}$ is closed under cocones, then we have $\mathcal{M}^{\wedge}=\operatorname{\mathsf{thick}}\mathcal{M}$. 
\item $\mathcal{M}^{\vee}=\operatorname{\mathsf{thick}}\mathcal{M}\cap{}^{\perp}\mathcal{M}$ holds. 
In particular, $\mathcal{M}^{\vee}$ is closed under cocones. 
If $\mathcal{M}$ is closed under cones, then we have $\mathcal{M}^{\vee}=\operatorname{\mathsf{thick}}\mathcal{M}$. 
\end{enumerate}
\end{lemma}

We give a relationship between ${}^{\perp_{>n}}\mathcal{M}$ and $\mathcal{M}^{\wedge}_{n}$. 

\begin{lemma}\label{lem-wedge_perp}
Let $\mathcal{M}$ be a presilting subcategory of $\mathcal{C}$ and let $n$ be a non-negative integer. 
Then the following statements hold. 
\begin{enumerate}[\upshape(1)]
\item ${}^{\perp_{>n}}\mathcal{M}={}^{\perp_{>n}}(\mathcal{M}^{\wedge})$ and $\mathcal{M}^{\perp_{>n}}=(\mathcal{M}^{\vee})^{\perp_{>n}}$.  
\item $\mathcal{M}^{\wedge}_{n}=\mathcal{M}^{\wedge}\cap{}^{\perp_{>n}}\mathcal{M}=\mathcal{M}^{\wedge}\cap{}^{\perp_{>n}}(\mathcal{M}^{\wedge})$. 
\item $\mathcal{M}^{\vee}_{n}=\mathcal{M}^{\vee}\cap\mathcal{M}^{\perp_{>n}}=\mathcal{M}^{\vee}\cap(\mathcal{M}^{\vee})^{\perp_{>n}}$.
\end{enumerate}
\end{lemma}

\begin{proof}
(1) Since the proof is similar, we only prove ${}^{\perp_{>n}}\mathcal{M}={}^{\perp_{>n}}(\mathcal{M}^{\wedge})$. 
By induction on $l$, we show ${}^{\perp_{>n}}\mathcal{M}\subseteq{}^{\perp_{>n}}(\mathcal{M}_{l}^{\wedge})$. 
If $l=0$, then this is clear. 
Assume $l\geq 1$.
Then the induction hypothesis implies that 
\begin{align}
\mathcal{M}^{\wedge}_{l}=\operatorname{\mathsf{cone}}(\mathcal{M}_{l-1}^{\wedge},\mathcal{M})\subseteq \operatorname{\mathsf{cone}}(({}^{\perp_{>n}}\mathcal{M})^{\perp_{>n}},({}^{\perp_{>n}}\mathcal{M})^{\perp_{>n}})\subseteq ({}^{\perp_{>n}}\mathcal{M})^{\perp_{>n}}.\notag
\end{align}
Hence we have ${}^{\perp_{>n}}\mathcal{M}\subseteq {}^{\perp_{>n}}(\mathcal{M}^{\wedge}_{l})$. 
Since the converse inclusion clearly holds, we obtain the assertion. 

(2) We show $\mathcal{M}^{\wedge}_{n}\subseteq{}^{\perp_{>n}}\mathcal{M}$ by induction on $n$. 
If $n=0$, then this is clear. 
Assume $n\geq 1$. 
Then the induction hypothesis implies
\begin{align}
\mathcal{M}^{\wedge}_{n}=\operatorname{\mathsf{cone}}(\mathcal{M}^{\wedge}_{n-1},\mathcal{M})\subseteq \operatorname{\mathsf{cone}}({}^{\perp_{>n-1}}\mathcal{M},{}^{\perp}\mathcal{M})\subseteq{}^{\perp_{>n}}\mathcal{M}. 
\notag 
\end{align} 
We prove the converse inclusion by induction on $n$. 
If $n=0$, then we have $\mathcal{M}^{\wedge}\cap{}^{\perp_{>0}}\mathcal{M}=\mathcal{M}^{\wedge}\cap{}^{\perp_{>0}}(\mathcal{M}^{\wedge})=\mathcal{M}$, where the first equality follows from (1).
Assume $n\geq1$. 
Then we have
\begin{align}
\mathcal{M}^{\wedge}\cap{}^{\perp_{>n}}\mathcal{M}=\operatorname{\mathsf{cone}}(\mathcal{M}^{\wedge},\mathcal{M})\cap{}^{\perp_{>n}}\mathcal{M}\subseteq\operatorname{\mathsf{cone}}(\mathcal{M}^{\wedge}\cap{}^{\perp_{>n-1}}\mathcal{M},\mathcal{M})\subseteq \operatorname{\mathsf{cone}}(\mathcal{M}^{\wedge}_{n-1},\mathcal{M}),\notag
\end{align} 
where the last inclusion follows from the induction hypothesis.
Thus the assertion holds. 

(3) By an argument similar to (2), we have the assertion. 
\end{proof}

Assume that $\mathcal{C}$ is a Krull--Schmidt category. 
Then $\mathcal{C}$ becomes a weakly idempotent complete category. 
Hence by \cite[Proposition 2.7]{K}, if $h=gf$ is an $\mathfrak{s}$-inflation, then so is $f$. 
Dually, if $h$ is an $\mathfrak{s}$-deflation, then so is $g$. 
Recall that any morphism $f: M\rightarrow N$ in $\mathcal{C}$ has two decompositions $f=\left[\begin{smallmatrix}f'\\0 \end{smallmatrix}\right]: M\to N'\oplus N''$ with $f'$ left minimal and $f=\left[\begin{smallmatrix}0& f'' \end{smallmatrix}\right]: M'\oplus M''\to N$ with $f''$ right minimal. 
Hence if $f$ is an $\mathfrak{s}$-inflation, then so is $f'$. 
Similarly, if $f$ is an $\mathfrak{s}$-deflation, then so is $f''$.
An object in $\mathcal{M}^{\wedge}_{l}$ admits the following $\mathfrak{s}$-conflations.  

\begin{lemma}\label{lem-sconfls}
Assume that $\mathcal{C}$ is a Krull--Schmidt category. 
Let $\mathcal{M}$ be a presilting subcategory of $\mathcal{C}$ and let $N$ be an object in $\mathcal{C}$. 
For a positive integer $l$, the following statements hold.
\begin{enumerate}[\upshape(1)]
\item $N\in \mathcal{M}^{\wedge}_{l}$ if and only if there exist $\mathfrak{s}$-conflations 
\begin{align}
\xymatrix@R=1mm{
N_{1}\ar[r]^-{g_{1}}&M_{0}\ar[r]^-{f_{0}}&N_{0}:=N\ar@{-->}[r]&\\
N_{2}\ar[r]^-{g_{2}}&M_{1}\ar[r]^-{f_{1}}&N_{1}\ar@{-->}[r]&\\
&\vdots&&\\
N_{l}\ar[r]^-{g_{l}}&M_{l-1}\ar[r]^-{f_{l-1}}&N_{l-1}\ar@{-->}[r]&\\
0\ar[r]^-{g_{l+1}}&M_{l}\ar[r]^-{f_{l}}&N_{l}\ar@{-->}[r]&
}\notag
\end{align}
such that $f_{i}$ is a minimal right $\mathcal{M}$-approximation and $g_{i+1}$ is in the Jacobson radical of $\mathcal{C}$ for each $i$. 
In this case, $N\notin \mathcal{M}^{\wedge}_{l-1}$ if and only if $M_{l}\neq 0$.
\item Assume that $N'$ is an object in $\mathcal{M}^{\wedge}$, that is, there exists an $\mathfrak{s}$-conflation $N'_{1}\xrightarrow{g'_{1}} M'_{0}\xrightarrow{f'_{0}} N'\dashrightarrow$ such that $N'_{1}\in\mathcal{M}^{\wedge}$, $f'_{0}$ is a minimal right $\mathcal{M}$-approximation and $g'_{1}$ is in the Jacobson radical of $\mathcal{C}$. 
If $\mathbb{E}^{l}(N,N')=0$, then $\operatorname{\mathsf{add}}M'_{0}\cap \operatorname{\mathsf{add}}M_{l}=\{0\}$.
\end{enumerate}
\end{lemma}

Note that for an $\mathfrak{s}$-conflation $L\xrightarrow{g}M\xrightarrow{f}N\dashrightarrow$, the following statements hold.
\begin{enumerate}[\upshape(a)]
\item $f$ is a right minimal morphism if and only if $g$ belongs to the Jacobson radical of $\mathcal{C}$.
\item $g$ is a left minimal morphism if and only if $f$ belongs to the Jacobson radical of $\mathcal{C}$.
\end{enumerate}

\begin{proof}
(1) Since the ``if'' part is clear, it is enough to show the ``only if'' part. 
By $N\in\mathcal{M}^{\wedge}_{l}$, there exists an $\mathfrak{s}$-conflation $N_{1}\xrightarrow{g_{1}}M_{0}\xrightarrow{f_{0}}N\dashrightarrow$ such that $N_{1}\in \mathcal{M}^{\wedge}_{l-1}$ and $M_{0}\in\mathcal{M}$. 
Since $\mathcal{C}$ is a Krull--Schmidt category and $\mathcal{M}^{\wedge}_{l-1}$ is closed under direct summands, we can assume that $f_{0}$ is right minimal. 
Since $\mathcal{M}\subseteq {}^{\perp}\mathcal{M}={}^{\perp}(\mathcal{M}^{\wedge})$ holds by  Lemma \ref{lem-wedge_perp}(1), $f_{0}$ is a right $\mathcal{M}$-approximation of $N$. 
By repeating this process, we obtain the desired $\mathfrak{s}$-conflations. 

(2) It is enough to show $\mathcal{C}(M_{l},M'_{0})=\mathrm{rad}_{\mathcal{C}}(M_{l},M'_{0})$.
Let $a\in \mathcal{C}(M_{l},M'_{0})$.
We show that $\mathcal{C}(g_{l},N')$ is surjective.
Applying $\mathcal{C}(-,N')$ to the $\mathfrak{s}$-conflation $M_{l}\xrightarrow{g_{l}}M_{l-1}\xrightarrow{f_{l-1}}N_{l-1}\dashrightarrow$ in (1), we have an exact sequence
\begin{align}
\mathcal{C}(M_{l-1},N')\xrightarrow{\mathcal{C}(g_{l},N')} \mathcal{C}(M_{l},N')\rightarrow \mathbb{E}(N_{l-1},N').\notag
\end{align}
By Lemma \ref{lem-psilt_wedge}(4), we have $\mathbb{E}^{i}(\mathcal{M},N')=0$ for each $i\geq 1$. 
Applying $\mathcal{C}(-,N')$ to the $\mathfrak{s}$-conflations in (1) induces isomorphisms
\begin{align}
\mathbb{E}(N_{l-1},N')\cong \cdots \cong \mathbb{E}^{l}(N_{0},N')=0,\notag
\end{align}
where the last equality follows from our assumption. 
Thus there exists a morphism $b\in \mathcal{C}(M_{l-1},N')$ such that $f'_{0}a=bg_{l}$.
Since $f'_{0}$ is a right $\mathcal{M}$-approximation, we have a morphism $b'\in \mathcal{C}(M_{l-1},M'_{0})$ with $b=f'_{0}b'$. 
By $f'_{0}(a-b'g_{l})=f'_{0}a-f'_{0}b'g_{l}=f'_{0}a-bg_{l}=0$, there exists a morphism $a'\in \mathcal{C}(M_{l}, N'_{1})$ such that $a-b'g_{l}=g'_{1}a'$.
Since $g'_{1}$ and $g_{l}$ are in the Jacobson radical of $\mathcal{C}$, we obtain $a=g'_{1}a'+b'g_{l}\in \mathrm{rad}_{\mathcal{C}}(M_{l},M'_{0})$.
\begin{align}
\xymatrix{
&M_{l}\ar[r]^-{g_{l}}\ar[d]^-{a}\ar[dl]_-{a'}&M_{l-1}\ar[r]\ar[d]^-{b}\ar[dl]_-{b'}&N_{l-1}\ar@{..>}[r]&\\
N'_{1}\ar[r]^-{g'_{1}}&M'_{0}\ar[r]^-{f'_{0}}&N'\ar@{..>}[r]&.&&
}\notag
\end{align}
This completes the proof.
\end{proof}

Note that $\mathfrak{s}$-conflations in Lemma \ref{lem-sconfls}(1) are uniquely determined up to isomorphisms. 
Hence we write such an $\mathfrak{s}$-conflation as $N_{i+1}\xrightarrow{g^{N}_{i+1}}M^{N}_{i}\xrightarrow{f^{N}_{i}}N_{i}\dashrightarrow$.

\subsection{Relationship between silting subcategories and hereditary cotorsion pairs}

In this subsection, we recall a relationship between silting subcategories and hereditary cotorsion pairs.
Let us recall the definition of hereditary cotorsion pairs, which are a common generalization of co-$t$-structures in a triangulated category and (complete) hereditary cotorsion pairs in an exact category. 

\begin{definition}\label{def-hcotors}
Let $\mathcal{X},\mathcal{Y}$ be subcategories of $\mathcal{C}$.
We call a pair $(\mathcal{X},\mathcal{Y})$  a \emph{hereditary cotorsion pair} in $\mathcal{C}$ if it satisfies the following conditions.
\begin{enumerate}[\upshape(1)]
\item $\mathcal{X}$ and $\mathcal{Y}$ are closed under direct summands.
\item $\mathbb{E}^{k}(\mathcal{X}, \mathcal{Y})=0$ for all $k\geq 1$.
\item $\mathcal{C}=\operatorname{\mathsf{cone}} (\mathcal{Y}, \mathcal{X})$.
\item $\mathcal{C}=\operatorname{\mathsf{cocone}} (\mathcal{Y}, \mathcal{X})$.
\end{enumerate}
\end{definition}

Let $\operatorname{\mathsf{hcotors}}\mathcal{C}$ denote the set of hereditary cotorsion pairs in $\mathcal{C}$.  
We write $(\mathcal{X}_{1},\mathcal{Y}_{1})\geq (\mathcal{X}_{2},\mathcal{Y}_{2})$ if $\mathcal{Y}_{1}\supseteq\mathcal{Y}_{2}$.
Then $(\operatorname{\mathsf{hcotors}}\mathcal{C}, \geq)$ clearly becomes a partially ordered set.

The following lemma is an easy observation.
 
\begin{lemma}\label{lem_cotors}
Let $(\mathcal{X}, \mathcal{Y})$ be a hereditary cotorsion pair in $\mathcal{C}$.
For each $n\geq 0$, we have $\mathcal{X}^{\wedge}_{n}={}^{\perp_{>n}}\mathcal{Y}$ and $\mathcal{Y}^{\vee}_{n}=\mathcal{X}^{\perp_{>n}}$.
In particular, $\mathcal{X}$ is closed under extensions and cocones, and $\mathcal{Y}$ is closed under extensions and cones.
\end{lemma}

\begin{proof}
Let $M\in{}^{\perp_{>n}}\mathcal{Y}$.
By $\mathcal{C}=\operatorname{\mathsf{cone}}(\mathcal{Y},\mathcal{X})$, there exists $Y^{1}\rightarrow X^{0}\rightarrow M\dashrightarrow$ such that $Y^{1}\in \mathcal{Y}$ and $X^{0}\in\mathcal{X}$.
Inductively, we have an $\mathfrak{s}$-conflation $Y^{i+1}\rightarrow X^{i}\rightarrow Y^{i}\dashrightarrow$ with $Y^{i},Y^{i+1}\in\mathcal{Y}$ and $X^{i}\in\mathcal{X}$.
Applying $\mathcal{C}(-,\mathcal{Y})$ to the $\mathfrak{s}$-conflations above induces isomorphisms 
\begin{align}
\mathbb{E}(Y^{n},\mathcal{Y})\cong\mathbb{E}^{2}(Y^{n-1},\mathcal{Y})\cong \cdots \cong \mathbb{E}^{n+1}(M,\mathcal{Y})=0, \notag
\end{align}
where the last equality follows from $M\in{}^{\perp_{>n}}\mathcal{Y}$.
Namely, the $\mathfrak{s}$-conflation $Y^{n+1}\rightarrow X^{n}\rightarrow Y^{n}\dashrightarrow$ splits. This implies $Y^{n}\in\mathcal{X}$.
Thus we obtain $M\in \mathcal{X}^{\wedge}_{n}$.
We show the converse inclusion by induction on $n$.
If $n=0$, then this is clear. Assume $n\geq 1$. By the induction hypothesis, we obtain 
\begin{align}
\mathcal{X}^{\wedge}_{n}=\operatorname{\mathsf{cone}}(\mathcal{X}^{\wedge}_{n-1},\mathcal{X})\subseteq \operatorname{\mathsf{cone}}({}^{\perp_{>n-1}}\mathcal{Y},{}^{\perp_{>0}}\mathcal{Y})\subseteq {}^{\perp_{>n}}\mathcal{Y}.\notag
\end{align}
Similarly, we obtain $\mathcal{Y}^{\vee}_{n}=\mathcal{X}^{\perp_{>n}}$.
\end{proof}

A hereditary cotorsion pair $(\mathcal{X}, \mathcal{Y})$ is said to be \emph{bounded} if $\mathcal{C}=\mathcal{X}^{\wedge}$ and $\mathcal{C}=\mathcal{Y}^{\vee}$.  
Let $\operatorname{\mathsf{bdd-hcotors}}\mathcal{C}$ denote the set of bounded hereditary cotorsion pairs in $\mathcal{C}$. 
Then bounded hereditary cotorsion pairs bijectively correspond silting subcategories as follows. 

\begin{proposition}[{\cite[Theorem 5.7]{AT22}}]\label{thm-AT22}
Let $\mathcal{C}$ be an extriangulated category.
Then there exist mutually inverse bijections
\begin{align}
\xymatrix{
\operatorname{\mathsf{silt}}\mathcal{C}\ar@<0.5ex>[r]^-{F} &\operatorname{\mathsf{bdd-hcotors}}\mathcal{C}\ar@<0.5ex>[l]^-{G}, 
}\notag\end{align}
where $F(\mathcal{M}):=(\mathcal{M}^{\vee},\mathcal{M}^{\wedge})$ and $G(\mathcal{X},\mathcal{Y}):=\mathcal{X}\cap\mathcal{Y}$.
\end{proposition}

For two subcategories $\mathcal{M},\mathcal{N}$ of $\mathcal{C}$, we write $\mathcal{M}\geq \mathcal{N}$ if $\mathbb{E}^{k}(\mathcal{M},\mathcal{N})=0$ for all $k\geq 1$. 

\begin{proposition}[{\cite[Proposition 5.12]{AT22}}]\label{prop-at512}
Let $\mathcal{M},\mathcal{N}\in\operatorname{\mathsf{silt}}\mathcal{C}$.
Then $\mathcal{M}\geq \mathcal{N}$ if and only if $\mathcal{M}^{\wedge}\supseteq\mathcal{N}^{\wedge}$.
In particular, $\geq$ gives a partial order on $\operatorname{\mathsf{silt}}\mathcal{C}$.
\end{proposition}

We give a basic property of the binary relation $\geq$.

\begin{lemma}\label{lem-order_wedge} 
Let $\mathcal{M}$, $\mathcal{N}$ be subcategories of $\mathcal{C}$. 
Then the following statements hold. 
\begin{enumerate}[\upshape(1)]
\item Assume that $\mathcal{M}$ is a silting subcategory of $\mathcal{C}$. 
Then $\mathcal{M}\geq \mathcal{N}$ if and only if $\mathcal{N}\subseteq \mathcal{M}^{\wedge}$.  
\item Assume that $\mathcal{N}$ is a silting subcategory of $\mathcal{C}$. 
Then $\mathcal{M}\geq \mathcal{N}$ if and only if $\mathcal{M}\subseteq \mathcal{N}^{\vee}$.
\item Assume that $\mathcal{M},\mathcal{N}$ are silting subcategories of $\mathcal{C}$. 
For each $n\geq 0$, $\mathcal{N}\subseteq \mathcal{M}^{\wedge}_{n}$ if and only if $\mathcal{M}\subseteq \mathcal{N}^{\vee}_{n}$.
\end{enumerate}
\end{lemma}

\begin{proof}
(1) Since $\mathcal{M}$ is silting, we have $\mathcal{C}=\operatorname{\mathsf{thick}}\mathcal{M}$.
By Lemma \ref{lem-psilt_wedge}(4), $\mathcal{M}^{\wedge}=\mathcal{M}^{\perp}$ holds.
Thus the assertion holds. 

(2) By an argument similar to (1), we have the assertion. 

(3) By (1) and (2), $\mathcal{N}\subseteq \mathcal{M}^{\wedge}$ if and only if $\mathcal{M}\subseteq \mathcal{N}^{\vee}$. 
Moreover, $\mathcal{N}\subseteq{}^{\perp_{>n}}\mathcal{M}$ if and only if $\mathcal{M}\subseteq\mathcal{N}^{\perp_{>n}}$ for each $n\geq 0$. 
Hence the assertion follows from Lemma \ref{lem-wedge_perp}(2) and (3). 
\end{proof}

\subsection{Silting subcategories of $\mathcal{P}^{\infty}$}

Let $\mathcal{P}^{\infty}:=(\operatorname{\mathsf{proj}}\mathcal{C})^{\wedge}$. 
Since $\mathcal{P}^{\infty}$ is closed under extensions, it becomes an extriangulated category by Example \ref{example-etcat}(3). 
In this subsection, we study silting subcategories of $\mathcal{P}^{\infty}$.
First, we recall the notion of projective dimension of subcategories of $\mathcal{C}$.
For a subcategory $\mathcal{X}$ of $\mathcal{C}$, we say that \emph{the projective dimension $\mathrm{pd}\mathcal{X}$ of $\mathcal{X}$ is at most $n$} if $\mathcal{X}\subseteq(\operatorname{\mathsf{proj}}\mathcal{C})^{\wedge}_{n}$. For an object $M$ in $\mathcal{C}$, $\mathrm{pd}M:=\mathrm{pd}(\operatorname{\mathsf{add}}M)$.
The projective dimension has a similar property to the projective dimension of modules.

\begin{lemma}[{\cite[Lemma 5.16]{AT22}}]\label{lem-516}
Let $\mathcal{X}$ be a subcategory of $\mathcal{C}$ and let $n$ be a non-negative integer. 
If $\mathrm{pd}\mathcal{X}\leq n$, then $\mathcal{C}=\mathcal{X}^{\perp_{>n}}$ holds. 
Moreover, if $\mathcal{C}$ has enough projective objects, then the converse also holds. 
\end{lemma}

Under a certain condition, we compare the projective dimension of $\mathcal{X}$ and that of $\mathcal{X}^{\vee}$.

\begin{lemma}\label{lem-pdim_vee}
Let $\mathcal{M}$ be a presilting subcategory of $\mathcal{C}$ and let $\mathcal{X}$ be a subcategory of $\mathcal{C}$. 
Assume that $\mathcal{M}$ is closed under cocones.
For each $n\geq 0$, $\mathcal{X}\subseteq \mathcal{M}^{\wedge}_{n}$ if and only if $\mathcal{X}^{\vee}\subseteq \mathcal{M}^{\wedge}_{n}$. 
Moreover, $\mathrm{pd}\mathcal{X}=\mathrm{pd}\mathcal{X}^{\vee}$ holds. 
\end{lemma}

\begin{proof}
Assume $\mathcal{X}\subseteq \mathcal{M}^{\wedge}_{n}$. 
For each $m\geq 0$, we have
\begin{align}
\mathcal{X}^{\vee}_{m}
&\subseteq (\mathcal{M}^{\wedge}_{n})^{\vee}_{m}&\notag\\
&= \operatorname{\mathsf{cone}}(\mathcal{M}^{\wedge}_{n-1},\mathcal{M}^{\vee}_{m})&\textnormal{by Lemma \ref{lem-psilt_wedge}(2)}\notag\\
&= \operatorname{\mathsf{cone}}(\mathcal{M}^{\wedge}_{n-1},\mathcal{M})&\textnormal{since $\mathcal{M}$ is closed under cocones} \notag\\
&=\mathcal{M}^{\wedge}_{n}&\textnormal{by definition}.\notag
\end{align}
Hence the former assertion holds. 
Since $\operatorname{\mathsf{proj}}\mathcal{C}$ is presilting and closed under cocones, the former assertion induces the latter assertion. 
\end{proof}

The following proposition gives a basic property of $\mathcal{P}^{\infty}$.

\begin{proposition}\label{prop-propertypfin}
We have $\mathcal{P}^{\infty}=\operatorname{\mathsf{thick}}(\operatorname{\mathsf{proj}}\mathcal{C})$.
Moreover, if $\mathcal{C}$ has enough projective objects, then $\mathcal{P}^{\infty}$ is a resolving subcategory of $\mathcal{C}$.
\end{proposition}

\begin{proof}
Since $\operatorname{\mathsf{proj}}\mathcal{C}$ is closed under cocones, we have $\mathcal{P}^{\infty}:=(\operatorname{\mathsf{proj}}\mathcal{C})^{\wedge}=\operatorname{\mathsf{thick}}(\operatorname{\mathsf{proj}}\mathcal{C})$, where the last equality follows from Lemma \ref{lem-psilt_wedge}(4).
Assume that $\mathcal{C}$ has enough projective objects. Then $\operatorname{\mathsf{proj}}\mathcal{C}\subseteq \mathcal{P}^{\infty}$ clearly holds.
Hence $\mathcal{P}^{\infty}$ is a resolving subcategory of $\mathcal{C}$.
\end{proof}

We give an example of silting subcategories of $\mathcal{P}^{\infty}$ and a characterization of presilting subcategories to be silting.

\begin{proposition}\label{prop-siltpfin}
The following statements hold. 
\begin{enumerate}[\upshape(1)]
\item $\operatorname{\mathsf{proj}}\mathcal{C}$ is a silting subcategory of $\mathcal{P}^{\infty}$. 
\item Let $\mathcal{M}$ be a presilting subcategory of $\mathcal{P}^{\infty}$. Then $\mathcal{M}$ is silting if and only if $\operatorname{\mathsf{proj}}\mathcal{C}\subseteq\mathcal{M}^{\vee}$.  
\end{enumerate}
\end{proposition}

\begin{proof}
(1) Since objects in $\operatorname{\mathsf{proj}}\mathcal{C}$ are also projective in $\mathcal{P}^{\infty}$, the subcategory $\operatorname{\mathsf{proj}}\mathcal{C}$ is presilting in $\mathcal{P}^{\infty}$.
Thus the assertion follows from Proposition \ref{prop-propertypfin}.

(2) We show that $\mathcal{P}^{\infty}=\operatorname{\mathsf{thick}}\mathcal{M}$ if and only if $\operatorname{\mathsf{proj}}\mathcal{C}\subseteq\mathcal{M}^{\vee}$. 
If $\operatorname{\mathsf{proj}}\mathcal{C}\subseteq\mathcal{M}^{\vee}$, then it follows from Lemma \ref{lem-psilt_wedge}(3) that $\mathcal{P}^{\infty}=(\operatorname{\mathsf{proj}}\mathcal{C})^{\wedge}\subseteq(\mathcal{M}^{\vee})^{\wedge}=\operatorname{\mathsf{thick}}\mathcal{M}$. 
Conversely, we assume $\mathcal{P}^{\infty}=\operatorname{\mathsf{thick}}\mathcal{M}$.  
Then we have $\operatorname{\mathsf{proj}}\mathcal{C}\subseteq\mathcal{P}^{\infty}\cap{}^{\perp}\mathcal{M}=\mathcal{M}^{\vee}$, where the last equality follows from Lemma \ref{lem-psilt_wedge}(5). 
\end{proof}

Using the notion of silting objects, we give a characterization of tilting modules over a finite dimensional algebra. 
Let $\Lambda$ be a finite dimensional algebra. By Example \ref{example-etcat}(1), $\operatorname{\mathsf{mod}}\Lambda$ forms an extriangulated category with enough projective objects.
Recall the definition of tilting modules. A $\Lambda$-module $T\in\operatorname{\mathsf{mod}}\Lambda$ is said to be \emph{tilting} if $\mathrm{Ext}_{\Lambda}^{i}(T,T)=0$ for all $i\geq 1$, $\mathrm{pd}T<\infty$ and $\Lambda\in (\operatorname{\mathsf{add}}T)^{\vee}$.

\begin{corollary}\label{ex-siltilt}
Let $\Lambda$ be a finite dimensional algebra and $\mathcal{P}^{\infty}:=(\operatorname{\mathsf{proj}}(\operatorname{\mathsf{mod}}\Lambda))^{\wedge}$.
Then silting objects of $\mathcal{P}^{\infty}$ coincide with tilting $\Lambda$-modules.
\end{corollary}

\begin{proof}
By Proposition \ref{prop-propertypfin}, $\mathcal{P}^{\infty}$ is a resolving subcategory of $\operatorname{\mathsf{mod}}\Lambda$.
Thus it follows from Lemma \ref{lem-thickE} that $T$ is a presilting object of $\mathcal{P}^{\infty}$ if and only if $T$ satisfies $\mathrm{Ext}_{\Lambda}^{i}(T,T)=0$ for all $i\geq 1$ and $\mathrm{pd}T<\infty$. 
Hence the assertion follows from Proposition \ref{prop-siltpfin}(2).
\end{proof}

Using Proposition \ref{thm-AT22}, we give a relationship between silting subcategories of $\mathcal{P}^{\infty}$ and hereditary cotorsion pairs in $\mathcal{P}^{\infty}$.

\begin{proposition}\label{prop-AT}
Assume that $\mathcal{C}$ has enough projective objects. 
Then there exist mutually inverse bijections
\begin{align}
\xymatrix{
\{\mathcal{M}\in\operatorname{\mathsf{silt}}\mathcal{P}^{\infty} \mid \mathrm{pd}\mathcal{M}<\infty\}\ar@<0.5ex>[r]^-{F} &\{(\mathcal{A},\mathcal{B})\in\operatorname{\mathsf{hcotors}}\mathcal{P}^{\infty}\mid \mathrm{pd}\mathcal{A}<\infty\}\ar@<0.5ex>[l]^-{G}, 
}\notag\end{align}
where $F(\mathcal{M}):=(\mathcal{M}^{\vee},\mathcal{M}^{\wedge})$ and $G(\mathcal{A},\mathcal{B}):=\mathcal{A}\cap\mathcal{B}$.
\end{proposition}

\begin{proof}
By Proposition \ref{thm-AT22}, we have mutually inverse bijections
\begin{align}\label{seq-basicbij}
\xymatrix{
\{\mathcal{M}\in\operatorname{\mathsf{silt}}\mathcal{P}^{\infty} \mid \mathrm{pd}\mathcal{M}<\infty\}\ar@<0.5ex>[r]^-{F} &\{(\mathcal{A},\mathcal{B})\in\operatorname{\mathsf{bdd-hcotors}}\mathcal{P}^{\infty}\mid \mathrm{pd}(\mathcal{A}\cap\mathcal{B})<\infty\}\ar@<0.5ex>[l]^-{G}, 
}\end{align}
where $F(\mathcal{M}):=(\mathcal{M}^{\vee},\mathcal{M}^{\wedge})$ and $G(\mathcal{A},\mathcal{B}):=\mathcal{A}\cap\mathcal{B}$.
It is enough to show that $(\mathcal{A},\mathcal{B})$ is a hereditary cotorsion pair in $\mathcal{P}^{\infty}$ with $\mathrm{pd}\mathcal{A}<\infty$ if and only if $(\mathcal{A},\mathcal{B})$ is a bounded hereditary cotorsion pair in $\mathcal{P}^{\infty}$ with $\mathrm{pd}(\mathcal{A}\cap \mathcal{B})<\infty$.
We show the ``if'' part. 
By \eqref{seq-basicbij}, $\mathcal{A}=(\mathcal{A}\cap \mathcal{B})^{\vee}$. 
Thus it follows from Lemma \ref{lem-pdim_vee} that $\mathrm{pd}\mathcal{A}=\mathrm{pd}(\mathcal{A}\cap\mathcal{B})^{\vee}=\mathrm{pd}(\mathcal{A}\cap\mathcal{B})<\infty$. 
We show the ``only if'' part.
By definition, $\mathrm{pd}(\mathcal{A}\cap\mathcal{B})\leq \mathrm{pd}\mathcal{A}<\infty$.
Since $\operatorname{\mathsf{proj}}\mathcal{C}$ is a subcategory of $\mathcal{A}$, we obtain $\mathcal{A}^{\wedge}=\mathcal{P}^{\infty}$.
Let $n:=\mathrm{pd}\mathcal{A}$. 
By Lemmas \ref{lem_cotors} and \ref{lem-516}, we have $\mathcal{P}^{\infty}=\mathcal{A}^{\perp_{>n}}\cap\mathcal{P}^{\infty}=\mathcal{B}^{\vee}_{n}$.
Thus $(\mathcal{A},\mathcal{B})$ is bounded.
This completes the proof.
\end{proof}

\section{Main results}

In this section, we study properties of silting subcategories following \cite{AI12}.
Let $\mathcal{C}$ be an extriangulated category.
Throughout this section, we assume that $\mathcal{C}$ has a silting subcategory.

\subsection{Grothendieck groups}

Assume that $\mathcal{C}$ is a Krull--Schmidt category. 
In this subsection, we give a description of the Grothendieck group of $\mathcal{C}$. 
For a subcategory $\mathcal{X}$ of $\mathcal{C}$, let $\operatorname{\mathsf{iso}}\mathcal{X}$ denote the set of isomorphism classes of objects in $\mathcal{X}$ and let $\operatorname{\mathsf{ind}}\mathcal{X}$ denote the set of isomorphism classes of indecomposable objects in $\mathcal{X}$.
Let $\mathbb{Z}^{\operatorname{\mathsf{ind}}\mathcal{X}}$ denote a free abelian group with a basis $\operatorname{\mathsf{ind}}\mathcal{X}$ and let $\mathbb{Z}^{\operatorname{\mathsf{ind}}\mathcal{X}}_{\geq 0}$ denote a free monoid with a basis $\operatorname{\mathsf{ind}}\mathcal{X}$. 
The \emph{Grothendieck group} $K_{0}(\mathcal{C})$ of $\mathcal{C}$ is defined as the quotient group $\mathbb{Z}^{\operatorname{\mathsf{ind}}\mathcal{C}}/R(\mathcal{C})$, where $R(\mathcal{C})$ is the subgroup of $\mathbb{Z}^{\operatorname{\mathsf{ind}}\mathcal{C}}$ generated by 
\begin{align}
\{[X]-[Y]+[Z]\mid \textnormal{there exists an $\mathfrak{s}$-conflation\;}X\rightarrow Y\rightarrow Z \dashrightarrow\textnormal{in\;}\mathcal{C}\}. \notag
\end{align}

The aim of this subsection is to show the following theorem. 

\begin{theorem}\label{mainresult-Grgroup}
Let $\mathcal{C}$ be a Krull--Schmidt extriangulated category.
If $\mathcal{C}$ has a silting subcategory $\mathcal{M}$, then the Grothendieck group $K_{0}(\mathcal{C})$ of $\mathcal{C}$ is a free abelian group with a basis $\operatorname{\mathsf{ind}}\mathcal{M}$.
\end{theorem}

For an object $M\in \mathcal{C}$, let $|M|$ denote the number of non-isomorphic indecomposable direct summands of $M$.
As a direct consequence of Theorem \ref{mainresult-Grgroup}, we obtain the following result.

\begin{corollary}\label{cor-Grgroup}
Assume that $\mathcal{C}$ has a silting object $M$.
Then the following statements hold.
\begin{enumerate}[\upshape(1)]
\item If $N$ is a silting object of $\mathcal{C}$, then we have $|N|=|M|$.
\item Let $N$ be a partial silting object of $\mathcal{C}$.
Then $N$ is silting if and only if $|N|=|M|$.
\end{enumerate}
\end{corollary}

\begin{proof}
(1) This follows from Theorem \ref{mainresult-Grgroup}.

(2) The ``only if'' part follows from (1).
We show the ``if'' part. Suppose that $N$ is not silting.
Since $N$ is a partial silting object, there exists an object $X$ such that $X\oplus N$ is silting and $X\notin \operatorname{\mathsf{add}}N$. Hence we have $|M|=|N|<|X|+|N|=|M|$, a contradiction.
\end{proof}

In the rest of this subsection, we give a proof of Theorem \ref{mainresult-Grgroup}.
Let $\mathcal{M}$ be a silting subcategory of $\mathcal{C}$. 
By Proposition \ref{thm-AT22}, we have $\mathcal{C}=\operatorname{\mathsf{cone}}(\mathcal{M}^{\wedge},\mathcal{M}^{\vee})=\operatorname{\mathsf{cocone}}(\mathcal{M}^{\wedge},\mathcal{M}^{\vee})$ and $\mathbb{E}^{i}(\mathcal{M}^{\vee}, \mathcal{M}^{\wedge})=0$ for all $i\geq 1$. 
To show the elements of $\operatorname{\mathsf{ind}}\mathcal{M}$ are linear independent, we define a map
\begin{align}
\gamma: \operatorname{\mathsf{iso}}\mathcal{C}\rightarrow\mathbb{Z}^{\operatorname{\mathsf{ind}}\mathcal{M}}\notag
\end{align}
as follows:
\begin{enumerate}[\upshape(i)]
\item $\gamma$ naturally identifies $\operatorname{\mathsf{iso}}\mathcal{M}$ with $\mathbb{Z}^{\operatorname{\mathsf{ind}}\mathcal{M}}_{\geq 0}$.
\item For each $N\in \mathcal{M}^{\wedge}_{l}$, there exists an $\mathfrak{s}$-conflation $N'\rightarrow M\xrightarrow{f}N\dashrightarrow$ such that $N'\in \mathcal{M}^{\wedge}_{l-1}$ and $f$ is a minimal right $\mathcal{M}$-approximation by Lemma \ref{lem-sconfls}(1). 
Then we inductively define 
\begin{align}
\displaystyle \gamma(N):=\gamma(M)-\gamma(N'). \notag
\end{align} 
\item For each $N\in \mathcal{M}^{\vee}_{l}$, there exists an $\mathfrak{s}$-conflation $N\xrightarrow{g}M\rightarrow N'\dashrightarrow$ such that $N'\in \mathcal{M}^{\vee}_{l-1}$ and $g$ is a minimal left $\mathcal{M}$-approximation by a dual statement of Lemma \ref{lem-sconfls}(1). 
Then we inductively define 
\begin{align}
\displaystyle \gamma(N):=\gamma(M)-\gamma(N'). \notag
\end{align} 
\item Let $N\in \mathcal{C}=\operatorname{\mathsf{cone}}(\mathcal{M}^{\wedge},\mathcal{M}^{\vee})$. 
Since $\mathcal{C}$ is a Krull--Schmidt category and $\mathcal{M}^{\wedge}$, $\mathcal{M}^{\vee}$ are closed under direct summands, it follows from $\mathbb{E}(\mathcal{M}^{\vee}, \mathcal{M}^{\wedge})=0$ that there exists an $\mathfrak{s}$-conflation $U_{N}\rightarrow V_{N}\xrightarrow{f} N\dashrightarrow$ such that $U_{N}\in \mathcal{M}^{\wedge}$, $V_{N}\in \mathcal{M}^{\vee}$ and $f$ is a minimal right $\mathcal{M}^{\vee}$-approximation. 
Then we put 
\begin{align}
\gamma(N):=\gamma(V_{N})-\gamma(U_{N}). \notag
\end{align}
\end{enumerate}

In the following, we show that the map $\gamma$ is additive, that is, $\gamma(X)-\gamma(Y)+\gamma(Z)=0$ holds for each $\mathfrak{s}$-conflation $X\rightarrow Y\rightarrow Z\dashrightarrow$. 

\begin{lemma}\label{lem-gamma_dirsummand}
Let $X_{1},X_{2}\in\mathcal{C}$. Then we have $\gamma(X_{1}\oplus X_{2})=\gamma(X_{1})+\gamma(X_{2})$.
\end{lemma}

\begin{proof}
First, we show by induction on $l$ that if $X_{1},X_{2}\in\mathcal{M}^{\wedge}_{l}$, then $\gamma(X_{1}\oplus X_{2})=\gamma(X_{1})+\gamma(X_{2})$ holds. 
If $l=0$, then the assertion follows from (i). 
Assume $l\geq 1$. 
By Lemma \ref{lem-sconfls}(1), there exists an $\mathfrak{s}$-conflation $K_{i}\rightarrow M_{i}\xrightarrow{\alpha_{i}}X_{i}\dashrightarrow$ such that $K_{i}\in\mathcal{M}^{\wedge}_{l-1}$ and $\alpha_{i}$ is a minimal right $\mathcal{M}$-approximation for $i=1,2$. 
Since $K_{1}\oplus K_{2}\in\mathcal{M}^{\wedge}_{l-1}$ and $\left[\begin{smallmatrix}\alpha_{1}&0\\0&\alpha_{2} \end{smallmatrix}\right]$ is also a minimal right $\mathcal{M}$-approximation, we  have $\gamma(X_{1}\oplus X_{2})=\gamma(M_{1}\oplus M_{2})-\gamma(K_{1}\oplus K_{2})$. 
Hence the induction hypothesis implies $\gamma(X_{1}\oplus X_{2})=\gamma(M_{1})+\gamma(M_{2})-\gamma(K_{1})-\gamma(K_{2})=\gamma(X_{1})+\gamma(X_{2})$. 
Similarly, if $X_{1},X_{2}\in\mathcal{M}^{\vee}_{l}$, then $\gamma(X_{1}\oplus X_{2})=\gamma(X_{1})+\gamma(X_{2})$ holds.
Let $X_{1},X_{2}\in\mathcal{C}$. 
There exists an $\mathfrak{s}$-conflation $U_{i}\rightarrow V_{i}\xrightarrow{f_{i}}X_{i}\dashrightarrow$ such that $U_{i}\in\mathcal{M}^{\wedge}$ and $f_{i}$ is a minimal right $\mathcal{M}^{\vee}$-approximation. Since $U_{1}\oplus U_{2}\in\mathcal{M}^{\wedge}$ and $\left[\begin{smallmatrix}f_{1}&0\\0&f_{2} \end{smallmatrix}\right]$ is also a minimal right $\mathcal{M}^{\vee}$-approximation, we  have $\gamma(X_{1}\oplus X_{2})=\gamma(V_{1}\oplus V_{2})-\gamma(U_{1}\oplus U_{2})$. 
By the arguments above, $\gamma(X_{1}\oplus X_{2})=\gamma(V_{1})+\gamma(V_{2})-\gamma(U_{1})-\gamma(U_{2})=\gamma(X_{1})+\gamma(X_{2})$ holds.
\end{proof}

We frequently use the following observation.

\begin{remark}
Let $N\in \mathcal{C}$.
\begin{enumerate}[\upshape(1)]
\item The map $\gamma$ does not depend on the minimality of morphisms. 
For example, if $U\rightarrow V\xrightarrow{f} N\dashrightarrow$ is an $\mathfrak{s}$-conflation with $U\in \mathcal{M}^{\wedge}_{l-1}$ and $V\in\mathcal{M}$, then we have $\gamma(N)=\gamma(V)-\gamma(U)$. Indeed, since $\mathcal{C}$ is a Krull--Schmidt category, we have an isomorphism of $\mathfrak{s}$-conflations
\begin{align}
\xymatrix{
U\ar[r]\ar[d]^{\cong}&V\ar[r]^{f}\ar[d]^{\cong}&N\ar@{-->}[r]\ar@{=}[d]&\\
N'\oplus W\ar[r]&M\oplus W\ar[r]^-{\left[\begin{smallmatrix}f'& 0\end{smallmatrix}\right]}&N\ar@{-->}[r]&,
}\notag
\end{align}
where $f'$ is right minimal. By Lemma \ref{lem-gamma_dirsummand}, we have
\begin{align}
\gamma(V)-\gamma(U)=\gamma(M\oplus W)-\gamma(N'\oplus W)=\gamma(M)+\gamma(W)-(\gamma(N')+\gamma(W))=\gamma(N).\notag
\end{align}
We have a similar observation for (iii) and (iv). 
\item By $\mathcal{C}=\operatorname{\mathsf{cocone}}(\mathcal{M}^{\wedge},\mathcal{M}^{\vee})$, there exists an $\mathfrak{s}$-conflation $N\xrightarrow{g} U'_{N}\rightarrow V'_{N}\dashrightarrow$ such that $U'_{N}\in \mathcal{M}^{\wedge}$,  $V'_{N}\in \mathcal{M}^{\vee}$ and $g$ is a minimal left $\mathcal{M}^{\wedge}$-approximation.
Then we may define $\gamma'(N)$ as $\gamma'(N):=\gamma(U'_{N})-\gamma(V'_{N})$.
We show $\gamma'(N)=\gamma(N)$.
By $U'_{N}\in \mathcal{M}^{\wedge}$, we have an $\mathfrak{s}$-conflation $U\rightarrow M\rightarrow U'_{N}\dashrightarrow$ with $M\in\mathcal{M}$ and $U\in \mathcal{M}^{\wedge}$.
Applying (ET4)$^{\mathrm{op}}$ to the $\mathfrak{s}$-conflations above induces a commutative diagram of $\mathfrak{s}$-conflations
\begin{align}
\xymatrix{
U\ar@{=}[r]\ar[d]&U\ar[d]&&\\
V\ar[r]\ar[d]&M\ar[r]\ar[d]&V'_{N}\ar@{-->}[r]\ar@{=}[d]&\\
N\ar[r]\ar@{-->}[d]&U'_{N}\ar[r]\ar@{-->}[d]&V'_{N}\ar@{-->}[r]&\\
&&&.
}\notag
\end{align}
Since $\mathcal{M}^{\vee}$ is closed under cocones, $V\in\mathcal{M}^{\vee}$ holds.
Thus we obtain 
\begin{align}
\gamma'(N)=\gamma(U'_{N})-\gamma(V'_{N})=(\gamma(M)-\gamma(U))-(\gamma(M)-\gamma(V))=\gamma(V)-\gamma(U)=\gamma(N).\notag
\end{align}
\end{enumerate}
\end{remark}

The map $\gamma$ satisfies the following properties.

\begin{lemma}\label{lemma-additive1532}
Let $X\rightarrow Y\rightarrow Z\dashrightarrow$ be an $\mathfrak{s}$-conflation.
Then the following statements hold.
\begin{enumerate}[\upshape(1)]
\item If $X,Y,Z\in \mathcal{M}^{\vee}$, then we have $\gamma(X)-\gamma(Y)+\gamma(Z)=0$.
\item If $X,Y,Z\in \mathcal{M}^{\wedge}$, then we have $\gamma(X)-\gamma(Y)+\gamma(Z)=0$.
\item If $Z\in \mathcal{M}^{\vee}$, then we have $\gamma(X)-\gamma(Y)+\gamma(Z)=0$.
\end{enumerate}
\end{lemma}

\begin{proof}
(1) For each $l\geq 0$, we consider the following statements.
\begin{enumerate}[\upshape(i)$_{l}$]
\item The statement holds true if $X,Y,Z\in \mathcal{M}^{\vee}_{l}$.
\item The statement holds true if $X\in \mathcal{M}^{\vee}_{l+1}$ and $Y,Z\in \mathcal{M}^{\vee}_{l}$.
\item The statement holds true if $X,Y\in \mathcal{M}^{\vee}_{l+1}$ and $Z\in \mathcal{M}^{\vee}_{l}$.
\end{enumerate}
First, we show that the statement (i)$_{0}$ holds true. 
Let $X\rightarrow Y\rightarrow Z\dashrightarrow$ be an $\mathfrak{s}$-conflation with $X,Y,Z\in \mathcal{M}$.
Then the $\mathfrak{s}$-conflation clearly splits. 
Hence we obtain $\gamma(Y)=\gamma(X)+\gamma(Z)$ by Lemma \ref{lem-gamma_dirsummand}.
 Next, we show that (i)$_{l}$ implies (ii)$_{l}$. Let $X\rightarrow Y\rightarrow Z\dashrightarrow$ be an $\mathfrak{s}$-conflation with $X\in \mathcal{M}^{\vee}_{l+1}$ and $Y,Z\in\mathcal{M}_{l}^{\vee}$.
By $X\in \mathcal{M}^{\vee}_{l+1}$, there exists an $\mathfrak{s}$-conflation $X\rightarrow M\rightarrow X'\dashrightarrow$ such that $M\in\mathcal{M}$ and $X'\in\mathcal{M}^{\vee}_{l}$. Thus we have $\gamma(X)=\gamma(M)-\gamma(X')$.
By a dual statement of \cite[Proposition 3.15]{NP19}, we obtain a commutative diagram
\begin{align}
\xymatrix{
X\ar[r]\ar[d]&Y\ar[r]\ar[d]&Z\ar@{-->}[r]\ar@{=}[d]&\\
M\ar[r]\ar[d]&E\ar[r]\ar[d]&Z\ar@{-->}[r]^{\delta}&\\
X'\ar@{=}[r]\ar@{-->}[d]&X'\ar@{-->}[d]&&\\
&&&.
}\notag
\end{align}
Since $\mathcal{M}^{\vee}_{l}$ is closed under extensions, we have $E\in \mathcal{M}^{\vee}_{l}$.
Hence the assumption (i)$_{l}$ induces $\gamma(Y)-\gamma(E)+\gamma(X')=0$.
On the other hand, Lemma \ref{lem-wedge_perp}(1) implies $\delta=0$, and hence $\gamma(E)=\gamma(M)+\gamma(Z)$ holds by Lemma \ref{lem-gamma_dirsummand}.
Therefore we obtain 
\begin{align}
\gamma(X)-\gamma(Y)+\gamma(Z)=(\gamma(M)-\gamma(X'))-(\gamma(E)-\gamma(X'))+(\gamma(E)-\gamma(M))=0.\notag
\end{align}
Similarly, we can show (ii)$_{l}\Rightarrow$(iii)$_{l}$ and (iii)$_{l}\Rightarrow$(i)$_{l+1}$.
Hence we have the assertion.

(2) By an argument similar to (1), we have the assertion.

(3) Let $X\rightarrow Y\rightarrow Z\dashrightarrow$ be an $\mathfrak{s}$-conflation with $Z\in \mathcal{M}^{\vee}$.
By $Y\in \mathcal{C}=\operatorname{\mathsf{cone}}(\mathcal{M}^{\wedge},\mathcal{M}^{\vee})$, there exists an $\mathfrak{s}$-conflation $U\rightarrow V\rightarrow Y\dashrightarrow$ such that $U\in\mathcal{M}^{\wedge}$ and $V\in \mathcal{M}^{\vee}$.
Thus $\gamma(Y)=\gamma(V)-\gamma(U)$.
By (ET4)$^{\mathrm{op}}$, we have a commutative diagram
\begin{align}
\xymatrix{
U\ar@{=}[r]\ar[d]&U\ar[d]&&\\
K\ar[r]\ar[d]&V\ar[r]\ar[d]&Z\ar@{-->}[r]\ar@{=}[d]&\\
X\ar[r]\ar@{-->}[d]&Y\ar[r]\ar@{-->}[d]&Z\ar@{-->}[r]&\\
&&&.
}\notag
\end{align}
Since $\mathcal{M}^{\vee}$ is closed under cocones, we have $K\in\mathcal{M}^{\vee}$.
Thus it follows from (1) that $\gamma(K)-\gamma(V)+\gamma(Z)=0$.
On the other hand, we have $\gamma(X)=\gamma(K)-\gamma(U)$.
Hence we obtain
\begin{align}
\gamma(X)-\gamma(Y)+\gamma(Z)=(\gamma(K)-\gamma(U))-(\gamma(V)-\gamma(U))+(\gamma(V)-\gamma(K))=0.\notag
\end{align}
This completes the proof.
\end{proof}

We obtain the desired property. 

\begin{lemma}\label{lemma-universal1556}
For each $\mathfrak{s}$-conflation $X\rightarrow Y\rightarrow Z\dashrightarrow$, we have $\gamma(X)-\gamma(Y)+\gamma(Z)=0$.
\end{lemma}

\begin{proof}
By $Z\in \mathcal{C}=\operatorname{\mathsf{cone}}(\mathcal{M}^{\wedge},\mathcal{M}^{\vee})$, there exists an $\mathfrak{s}$-conflation $U\rightarrow V\rightarrow Z\dashrightarrow$ such that $U\in\mathcal{M}^{\wedge}$ and $V\in\mathcal{M}^{\vee}$.
By \cite[Proposition 3.15]{NP19}, we have a commutative diagram
\begin{align}
\xymatrix{
&U\ar@{=}[r]\ar[d]&U\ar[d]&\\
X\ar[r]\ar@{=}[d]&E\ar[r]\ar[d]&V\ar@{-->}[r]\ar[d]&\\
X\ar[r]&Y\ar[r]\ar@{-->}[d]&Z\ar@{-->}[r]\ar@{-->}[d]&\\
&&&.
}\notag
\end{align}
Thus it follows from Lemma \ref{lemma-additive1532}(3) that $\gamma(X)-\gamma(E)+\gamma(V)=0$.
By $E\in\mathcal{C}=\operatorname{\mathsf{cone}}(\mathcal{M}^{\wedge},\mathcal{M}^{\vee})$, there exists an $\mathfrak{s}$-conflation $U'\rightarrow V'\rightarrow E\dashrightarrow$ such that $U'\in\mathcal{M}^{\wedge}$ and $V'\in\mathcal{M}^{\vee}$.
By (ET4)$^{\mathrm{op}}$, we obtain a commutative diagram
\begin{align}
\xymatrix{
U'\ar@{=}[r]\ar[d]&U'\ar[d]&&\\
K\ar[r]\ar[d]&V'\ar[r]\ar[d]&Y\ar@{-->}[r]\ar@{=}[d]&\\
U\ar[r]\ar@{-->}[d]&E\ar[r]\ar@{-->}[d]&Y\ar@{-->}[r]&\\
&&&.
}\notag
\end{align}
Since $\mathcal{M}^{\wedge}$ is closed under extensions, we have $K\in\mathcal{M}^{\wedge}$.
By Lemma \ref{lemma-additive1532}(2), we have $\gamma(U')-\gamma(K)+\gamma(U)=0$. 
Moreover, $\gamma(Y)=\gamma(V')-\gamma(K)$ holds.
Hence we obtain
\begin{align}
\gamma(X)-\gamma(Y)+\gamma(Z)
&=(\gamma(E)-\gamma(V))-(\gamma(V')-\gamma(K))+(\gamma(V)-\gamma(U))\notag\\
&=(\gamma(E)-\gamma(V'))+(\gamma(K)-\gamma(U))=-\gamma(U')+\gamma(U')=0\notag
\end{align}
This completes the proof.
\end{proof}

Now we are ready to prove Theorem \ref{mainresult-Grgroup}.

\begin{proof}[Proof of Theorem \ref{mainresult-Grgroup}]
First, we show that $\operatorname{\mathsf{ind}}\mathcal{M}$ generates $K_{0}(\mathcal{C})$.
For each object $N\in \mathcal{C}=\operatorname{\mathsf{cone}}(\mathcal{M}^{\wedge},\mathcal{M}^{\vee})$, there exists an $\mathfrak{s}$-conflation $U\rightarrow V\rightarrow N\dashrightarrow$ such that $U\in\mathcal{M}^{\wedge}$ and $V\in\mathcal{M}^{\vee}$.
By the definitions of $\mathcal{M}^{\wedge}$ and $\mathcal{M}^{\vee}$, the elements $[U], [V]\in K_{0}(\mathcal{C})$ are generated by $\operatorname{\mathsf{ind}}\mathcal{M}$.
Hence $[N]\in K_{0}(\mathcal{C})$ is also generated by $\operatorname{\mathsf{ind}}\mathcal{M}$.
Next, we show that each finite subset of $\operatorname{\mathsf{ind}}\mathcal{M}$ is linear independent. 
By Lemma \ref{lemma-universal1556}, $\gamma$ is additive. 
By the universality of Grothendieck groups, there exists a unique group homomorphism $K_{0}(\mathcal{C})\to \mathbb{Z}^{\operatorname{\mathsf{ind}}\mathcal{M}}$.
Hence $\operatorname{\mathsf{ind}}\mathcal{M}$ forms a basis of $K_{0}(\mathcal{C})$.
\end{proof}

\subsection{Silting mutation}

In this subsection, we introduce the notion of silting mutation in extriangulated categories and study its basic properties.
Recall the definition of covariantly finite subcategories of $\mathcal{C}$.

\begin{definition}
Let $\mathcal{M}$ be a subcategory of $\mathcal{C}$ and let $\mathcal{D}$ be a subcategory of $\mathcal{M}$.
\begin{enumerate}[\upshape(1)]
\item We call $\mathcal{D}$ a \emph{covariantly finite subcategory} of $\mathcal{M}$ if each object in $\mathcal{M}$ has a left $\mathcal{D}$-approximation. Dually, we define a \emph{contravariantly finite subcategory}.
\item A covariantly finite subcategory $\mathcal{D}$ is said to be \emph{good} if each object in $\mathcal{M}$ has a left $\mathcal{D}$-approximation which is an $\mathfrak{s}$-inflation (in $\mathcal{C}$).
Dually, we define a \emph{good contravariantly finite subcategory}.
\item We call $\mathcal{D}$  a \emph{functorially finite subcategory} of $\mathcal{M}$ if it is covariantly finite and contravariantly finite.
\end{enumerate}
\end{definition}

We give a sufficient condition for covariantly finite subcategories to be good.

\begin{lemma}\label{lem-good}
Assume that $\mathcal{C}$ is weakly idempotent complete.
Let $\mathcal{M}$ be a subcategory of $\mathcal{C}$ and let $\mathcal{D}$ be a subcategory of $\mathcal{M}$.
Then the following statements hold.
\begin{enumerate}[\upshape(1)]
\item Assume that $\mathcal{D}$ is a covariantly finite subcategory of $\mathcal{M}$.
If $\mathcal{M}\subseteq \operatorname{\mathsf{cocone}}(\mathcal{D},\mathcal{C})$, then $\mathcal{D}$ is good covariantly finite.
\item Assume that $\mathcal{D}$ is a contravariantly finite subcategory of $\mathcal{M}$. 
If $\mathcal{M}\subseteq\operatorname{\mathsf{cone}}(\mathcal{C},\mathcal{D})$, then $\mathcal{D}$ is good contravariantly finite.
\end{enumerate}
\end{lemma}

\begin{proof}
Since the proof is similar, we only show (1).
Let $M\in\mathcal{M}$.
By $\mathcal{M}\subseteq \operatorname{\mathsf{cocone}}(\mathcal{D},\mathcal{C})$, there exists an $\mathfrak{s}$-inflation $f:M\rightarrow D$ such that $D\in \mathcal{D}$.
Since $\mathcal{D}$ is a covariantly finite subcategory of $\mathcal{M}$, we obtain a left $\mathcal{D}$-approximation $f':M\rightarrow D'$.
Thus we have a morphism $g:D'\rightarrow D$ with $f=gf'$.
Since $\mathcal{C}$ is weakly idempotent complete, $f'$ is an $\mathfrak{s}$-inflation.
\end{proof}

We collect basic properties of good covariantly finite subcategories of a presilting subcategory.

\begin{lemma}\label{lemma-basicresult1628}
Let $\mathcal{M}$ be a presilting subcategory of $\mathcal{C}$ and let $\mathcal{D}$ be a good covariantly finite subcategory of $\mathcal{M}$.
Then for each $M\in\mathcal{M}$, there exists an $\mathfrak{s}$-conflation 
\begin{align}\label{basic-resolution}
M\xrightarrow{f}D\xrightarrow{g}N\dashrightarrow
\end{align}
such that $f$ is a left $\mathcal{D}$-approximation. Moreover, the following statements hold.
\begin{enumerate}[\upshape(1)]
\item $N\in \mathcal{M}^{\perp}$.
\item $N\in {}^{\perp_{>1}}\mathcal{M}$.
\item $N\in {}^{\perp}\mathcal{D}$.
\item $g$ is a right $\mathcal{D}$-approximation.
\end{enumerate}
\end{lemma}

\begin{proof}
Let $M\in\mathcal{M}$.
Since $\mathcal{D}$ is a good covariantly finite subcategory of $\mathcal{M}$, there exists a left $\mathcal{D}$-approximation $f:M\rightarrow D$ which is an $\mathfrak{s}$-inflation. Hence we have the $\mathfrak{s}$-conflation \eqref{basic-resolution}.
By \eqref{basic-resolution}, we obtain $N\in \mathcal{M}^{\wedge}_{1}$.
Thus the statements (1) and (2) follows from Lemmas \ref{lem-psilt_wedge}(4) and \ref{lem-wedge_perp}(2) respectively.
By $\mathcal{D}\subseteq \mathcal{M}$, we have $N\in {}^{\perp_{>1}}\mathcal{D}$.
Applying $\mathcal{C}(-,\mathcal{D})$ to \eqref{basic-resolution} induces an exact sequence 
\begin{align}
\mathcal{C}(D,\mathcal{D})\xrightarrow{\mathcal{C}(f,\mathcal{D})}\mathcal{C}(M,\mathcal{D})\rightarrow \mathbb{E}(N,\mathcal{D})\rightarrow \mathbb{E}(D,\mathcal{D})=0. \notag
\end{align}
Since $f$ is a left $\mathcal{D}$-approximation, we have $\mathbb{E}(N,\mathcal{D})=0$.
Hence the statement (3) holds true.
We show the statement (4). 
Applying $\mathcal{C}(\mathcal{D},-)$ to \eqref{basic-resolution}, we obtain an exact sequence 
\begin{align}
\mathcal{C}(\mathcal{D},D)\xrightarrow{\mathcal{C}(\mathcal{D},g)}\mathcal{C}(\mathcal{D},N)\rightarrow\mathbb{E}(\mathcal{D},M)=0.\notag
\end{align}
Hence $g$ is a right $\mathcal{D}$-approximation.
\end{proof}

Now, we introduce mutation of presilting subcategories of extriangulated categories.

\begin{definition}\label{def-mut}
Let $\mathcal{M}$ be a presilting subcategory of $\mathcal{C}$ and let $\mathcal{D}$ be a good covariantly finite subcategory of $\mathcal{M}$.
For each $M\in\mathcal{M}$, we take a left $\mathcal{D}$-approximation $f:M\to D$ and an $\mathfrak{s}$-conflation
\begin{align}
M\xrightarrow{f}D\rightarrow N_{M} \dashrightarrow . \notag
\end{align}
We define a subcategory $\mu^{L}(\mathcal{M};\mathcal{D})$ of $\mathcal{C}$ as 
\begin{align}
\mu^{L}(\mathcal{M};\mathcal{D}):=\operatorname{\mathsf{add}}(\mathcal{D}\cup \{ N_{M}\mid M\in\mathcal{M}\}) \notag
\end{align}
and call it a \emph{left mutation} of $\mathcal{M}$ with respect to $\mathcal{D}$.
Dually, we define a \emph{right mutation} $\mu^{R}(\mathcal{M};\mathcal{D})$ of $\mathcal{M}$ with respect to a good contravariantly finite subcategory $\mathcal{D}$.
\end{definition}

We give some remark on mutation.

\begin{remark}\label{remark-basic1845}
Let $\mathcal{M}$ be a presilting subcategory of $\mathcal{C}$ and let $\mathcal{D}$ be a good covariantly finite subcategory of $\mathcal{M}$.
\begin{enumerate}[\upshape(1)]
\item The left mutation $\mu^{L}(\mathcal{M};\mathcal{D})$ does not depend on a choice of a left $\mathcal{D}$-approximation.
Indeed, if there are two $\mathfrak{s}$-conflations $M\xrightarrow{f}D\rightarrow N\dashrightarrow$ and $M\xrightarrow{f'}D'\rightarrow N'\dashrightarrow$ such that $f, f'$ are left $\mathcal{D}$-approximations, then we have a commutative diagram
\begin{align}
\xymatrix{M\ar[r]^{f}\ar[d]_{f'}&D\ar[r]\ar[d]&N\ar@{-->}[r]\ar@{=}[d]&\\D'\ar[r]\ar[d]&E\ar[r]\ar[d]&N\ar@{-->}[r]&\\N'\ar@{=}[r]\ar@{-->}[d]&N'\ar@{-->}[d]&&\\&&&}\notag
\end{align}
by a dual statement of \cite[Proposition 3.15]{NP19}.
Hence Lemma \ref{lemma-basicresult1628}(3) implies $N\oplus D' \cong E\cong N'\oplus D$.
\item We have $\mu^{L}(\mathcal{M};\mathcal{D})=\mu^{L}(\mathcal{M};\operatorname{\mathsf{add}}\mathcal{D})$. Indeed, this follows from (1) and a property that each left $\mathcal{D}$-approximation is a left $\operatorname{\mathsf{add}}\mathcal{D}$-approximation. Thus we may assume $\mathcal{D}=\operatorname{\mathsf{add}}\mathcal{D}$ whenever $\mathcal{D}$ is a good covariantly finite subcategory.
\end{enumerate}
\end{remark}

We give basic properties of mutation of presilting subcategories.

\begin{theorem}\label{thm-main2-1845}
Let $\mathcal{M}$ be a presilting subcategory of $\mathcal{C}$ and let $\mathcal{D}$ be a good covariantly finite subcategory of $\mathcal{M}$.
Then the following statements hold.
\begin{enumerate}[\upshape(1)]
\item $\mu^{L}(\mathcal{M};\mathcal{D})$ is a presilting subcategory of $\mathcal{C}$.
\item $\mathcal{M}\ge \mu^{L}(\mathcal{M};\mathcal{D})$, where the equality holds if and only if $\mathcal{M}=\operatorname{\mathsf{add}}\mathcal{D}$.
\item If $\mathcal{M}$ is a silting subcategory of $\mathcal{C}$, then so is $\mu^{L}(\mathcal{M};\mathcal{D})$.
\end{enumerate}
\end{theorem}

\begin{proof}
Let $\mathcal{N}:=\mu^{L}(\mathcal{M};\mathcal{D})$.

(1) By definition, $\mathcal{N}=\operatorname{\mathsf{add}}\mathcal{N}$ holds. We show $\mathbb{E}^{k}(\mathcal{N},\mathcal{N})=0$ for all $k\ge 1$. 
Let $N\in\mathcal{N}$. Then there exists an $\mathfrak{s}$-conflation $M\xrightarrow{f}D\rightarrow N\dashrightarrow$ such that $M\in\mathcal{M}$ and $f$ is a left $\mathcal{D}$-approximation.
Applying $\mathcal{C}(\mathcal{N},-)$ to the $\mathfrak{s}$-conflation above induces an exact sequence
\begin{align}
\mathbb{E}^{k}(\mathcal{N},D)\rightarrow\mathbb{E}^{k}(\mathcal{N},N)\rightarrow\mathbb{E}^{k+1}(\mathcal{N},M).\notag
\end{align}
By Lemma \ref{lemma-basicresult1628}(2) and (3), both end sides vanish, and hence $\mathbb{E}^{k}(\mathcal{N},N)=0$ for all $k\geq1$.

(2) By Lemma \ref{lemma-basicresult1628}(1), $\mathcal{M}\ge \mathcal{N}$ holds.
We show that $\mathcal{M}=\mathcal{N}$ if and only if $\mathcal{M}=\operatorname{\mathsf{add}}\mathcal{D}$.
Assume that there exists an $\mathfrak{s}$-conflation $M\rightarrow D\rightarrow N\overset{\delta}{\dashrightarrow}$ such that $M\in\mathcal{M}$, $D\in\mathcal{D}$ and $N\in\mathcal{N}$.
If $M\notin\operatorname{\mathsf{add}}\mathcal{D}$ holds, then we obtain $\delta\neq 0$, and hence $\mathbb{E}(N,M)\neq 0$. 
This implies $N\notin \mathcal{M}$.
On the other hand, we assume $M\in \operatorname{\mathsf{add}}\mathcal{D}$. By Lemma \ref{lemma-basicresult1628}(3), we have $\delta=0$, and hence $N\in\operatorname{\mathsf{add}}\mathcal{D}$. 

(3) By (1), $\mathcal{N}$ is a presilting subcategory of $\mathcal{C}$. 
By the definition of left mutation, we obtain $\mathcal{M}\subseteq \mathcal{N}^{\vee}_{1}\subseteq \operatorname{\mathsf{thick}}\mathcal{N}$. 
Thus it follows from $\mathcal{C}=\operatorname{\mathsf{thick}}\mathcal{M}\subseteq \operatorname{\mathsf{thick}}\mathcal{N}$ that $\mathcal{N}$ is a silting subcategory of $\mathcal{C}$.
\end{proof}

Dually, if $\mathcal{D}$ is a good contravariantly finite subcategory of a silting subcategory $\mathcal{M}$, then we can show that $\mu^{R}(\mathcal{M};\mathcal{D})$ is a silting subcategory of $\mathcal{M}$ with $\mu^{R}(\mathcal{M};\mathcal{D})\geq \mathcal{M}$.
The following proposition is an easy observation.

\begin{proposition}
Let $\mathcal{M}$ be a presilting subcategory of $\mathcal{C}$ and let $\mathcal{D}$ be a good covariantly finite subcategory of $\mathcal{M}$.
Then the following statements hold.
\begin{enumerate}[\upshape(1)]
\item $\mathcal{D}$ is a good contravariantly finite subcategory of $\mu^{L}(\mathcal{M};\mathcal{D})$.
\item $\mu^{R}(\mu^{L}(\mathcal{M};\mathcal{D});\mathcal{D})=\mathcal{M}$.
\end{enumerate}
\end{proposition}

\begin{proof}
(1) This follows from Lemma \ref{lemma-basicresult1628}(4).

(2) Let $N\in \mu^{L}(\mathcal{M};\mathcal{D})$.
Then there exist two $\mathfrak{s}$-conflations $L\rightarrow D\xrightarrow{g}N\dashrightarrow$ and $M\rightarrow D'\xrightarrow{g'}N\rightarrow$ such that $L\in\mu^{R}(\mu^{L}(\mathcal{M};\mathcal{D});\mathcal{D})$, $M\in\mathcal{M}$, and $g,g'$ are right $\mathcal{D}$-approximations.
By a dual argument of Remark \ref{remark-basic1845}(1), we obtain $M\oplus D'\cong L\oplus D$.
Hence we have the assertion.
\end{proof}

Since a covariantly finite subcategory is not necessarily good, silting mutation in extriangulated categories is often impossible.
We give a sufficient condition for silting mutation to be possible. 
In the rest of this subsection, we assume that $R$ is a field and $\mathcal{C}$ is a Hom-finite Krull--Schmidt category. 
The following lemma plays an important role, which is an analog of \cite{AS80, IZ19}.

\begin{lemma}\label{lem-311} 
Assume that there exists an $\mathfrak{s}$-conflation $A\xrightarrow{f} A' \rightarrow C\dashrightarrow$ such that $f$ is in the Jacobson radical of $\mathcal{C}$. 
Then, for each indecomposable direct summand $X$ of $A$, there exists an $\mathfrak{s}$-conflation $X\rightarrow A_{X}\rightarrow C_{X}\dashrightarrow$ such that $X\notin\operatorname{\mathsf{add}}A_{X}$ and $A_{X}\in\operatorname{\mathsf{add}}A'$. 
\end{lemma}

\begin{proof}
Let $X$ be an indecomposable direct summand of $A$ and let $\iota: X\rightarrow A$ be the natural injection. 
Then $g:=f\iota$ is an $\mathfrak{s}$-inflation and in the Jacobson radical of $\mathcal{C}$. 
Decompose $A'=X^{\oplus l}\oplus U$ with $X\notin \operatorname{\mathsf{add}}U$. 
If $l=0$, then $X\xrightarrow{g}A'\rightarrow\mathrm{Cone}(g)\dashrightarrow$ is a desired $\mathfrak{s}$-conflation. 
Assume $l\geq 1$. 
Let $Y_{0}:=X$ and $g_{0}:=g$. 
For each $i \ge 1$, we inductively define
\begin{align}Y_{i}:=X^{\oplus l^{i}}\oplus U^{\oplus (1+l +\cdots +l^{i-1})}\notag\end{align}
 and 
 \begin{align}g_{i}:=g^{\oplus l^{i}}\oplus \mathrm{id}_{U}^{\oplus (1+l+\cdots +l^{i-1})}:Y_{i}\rightarrow Y_{i+1}.\notag \end{align}
 Since $\mathcal{C}$ is Hom-finite, we can take $m>0$ such that $\mathrm{rad}_{\mathcal{C}}^{m}(X,X)=0$. 
 Hence the composition morphism 
 \begin{align}g_{m-1}\cdots g_{1}g_{0}:X \rightarrow Y_{m}=X^{\oplus l^{m}}\oplus U^{\oplus(1+l+\cdots+ l^{m-1})}\notag \end{align}
is a direct sum of two morphisms $0: X\rightarrow X^{\oplus l^{m}}$ and $h: X\rightarrow U^{\oplus (1+l+\cdots +l^{m-1})}$. 
Since each $g_{i}$ is an $\mathfrak{s}$-inflation, so is $\left[\begin{smallmatrix}0\\ h\end{smallmatrix}\right]$. 
Since $\mathcal{C}$ is a Krull--Schmidt category, $h$ is an $\mathfrak{s}$-inflation. 
Thus we obtain a desired $\mathfrak{s}$-conflation $X\xrightarrow{h} U^{\oplus (1+l +\cdots+ l^{m-1})}\rightarrow \mathrm{Cone}(h) \dashrightarrow$. 
\end{proof}

Let $\mathcal{M}$ be a presilting subcategory of $\mathcal{C}$ and let $X\in\mathcal{M}$.
Define a subcategory $\mathcal{M}_{X}$ of $\mathcal{M}$ as $\mathcal{M}_{X}:=\operatorname{\mathsf{add}}(\operatorname{\mathsf{ind}}\mathcal{M}\setminus\operatorname{\mathsf{ind}}(\operatorname{\mathsf{add}}X))$.
If $\mathcal{M}_{X}$ is good covariantly finite, then we put $\mu_{X}(\mathcal{M}):=\mu^{L}(\mathcal{M};\mathcal{M}_{X})$.
We call $\mu_{X}(\mathcal{M})$ an \emph{irreducible left mutation} if $X$ is indecomposable.
In the following, we assume the following condition.
\begin{itemize}
\item[(F)] For each silting subcategory $\mathcal{M}$ and $X\in\operatorname{\mathsf{ind}}\mathcal{M}$, the subcategory $\mathcal{M}_{X}$ is a functorially finite subcategory of $\mathcal{M}$. 
\end{itemize}
If $\mathcal{C}$ has a silting object, then the condition (F) is satisfied.

\begin{proposition}\label{prop-mut}
Let $\mathcal{M}$ be a silting subcategory of $\mathcal{C}$ and let $\mathcal{N}$ be a presilting subcategory satisfying $\mathcal{M}>\mathcal{N}$ and $\mathcal{M}\not\supseteq\mathcal{N}$.
Then the following statements hold.
\begin{enumerate}[\upshape(1)]
\item There exists $N\in\mathcal{N}$ such that $N\in \mathcal{M}^{\wedge}_{l}\setminus \mathcal{M}^{\wedge}_{l-1}$ for some $l\geq 1$.
\item If $X$ is an indecomposable direct summand of $M^{N}_{l}$, then $\mathcal{M}_{X}$ is a good covariantly finite subcategory of $\mathcal{M}$.
\item We have $\mathcal{M}>\mu_{X}(\mathcal{M})\geq\mathcal{N}$.
\end{enumerate}
\end{proposition}

\begin{proof}
(1) By Lemma \ref{lem-order_wedge}(1), $\mathcal{N}\subseteq\mathcal{M}^{\wedge}$. 
Since $\mathcal{M}\neq\mathcal{N}$, there exists $N\in\mathcal{N}$ such that $N\in\mathcal{M}^{\wedge}_{l}\setminus \mathcal{M}^{\wedge}_{l-1}$ for some $l\geq1$. 

(2) By Lemma \ref{lem-sconfls}(1), there exists an $\mathfrak{s}$-conflaion $M^{N}_{l}\xrightarrow{g^{N}_{l}}M^{N}_{l-1}\rightarrow N_{l-1}\dashrightarrow$ such that $g^{N}_{l}$ is in the Jacobson radical of $\mathcal{C}$.
Let $X$ be an indecomposable direct summand of $M^{N}_{l}$.
By Lemma \ref{lem-311}, we obtain an $\mathfrak{s}$-conflation $X\xrightarrow{\iota} M_{X} \rightarrow C_{X} \dashrightarrow$ with $M_{X}\in \mathcal{M}_{X}$. Thus we obtain $\mathcal{M}\subseteq \operatorname{\mathsf{cocone}}(\mathcal{M}_{X},\mathcal{C})$.
By Lemma \ref{lem-good}(1), $\mathcal{M}_{X}$ is a good covariantly finite subcategory.

(3) Since $\mathcal{M}_{X}$ is a good covariantly finite subcategory, there exists an $\mathfrak{s}$-conflation $X\xrightarrow{\alpha} M'\rightarrow C \dashrightarrow$ such that $\alpha$ is a left $\mathcal{M}_{X}$-approximation. 
Let $\mathcal{L}:=\mu_{X}(\mathcal{M})=\operatorname{\mathsf{add}}(\mathcal{M}_{X}\cup \{C\})$.   
By Theorem \ref{thm-main2-1845}, $\mathcal{L}$ is a silting subcategory of $\mathcal{C}$ and $\mathcal{M}> \mathcal{L}$. 
We show $\mathcal{L}\geq \mathcal{N}$, or equivalently $\mathbb{E}^{i}(C,\mathcal{N})=0$ for each $i\geq 1$.
Applying $\mathcal{C}(-,\mathcal{N})$ to the $\mathfrak{s}$-conflation $X\xrightarrow{\alpha}M'\rightarrow C\dashrightarrow$ induces an exact sequence 
\begin{align}
\mathbb{E}^{i-1}(M',\mathcal{N})\rightarrow \mathbb{E}^{i-1}(X, \mathcal{N})\rightarrow \mathbb{E}^{i}(C, \mathcal{N})\rightarrow \mathbb{E}^{i}(M', \mathcal{N}).\notag 
\end{align}
If $i \geq 2$, then the exact sequence above implies $\mathbb{E}^{i}(C, \mathcal{N})=0$. 
Thus it is enough to show $\mathbb{E}(C, N')=0$ for each $N'\in\mathcal{N}$.
Let $a \in \mathcal{C}(X, N')$. 
By Lemma \ref{lem-sconfls}(1), there exists an $\mathfrak{s}$-conflation $N'_{1}\rightarrow M^{N'}_{0}\xrightarrow{f^{N'}_{0}} N'\dashrightarrow$ such that $f^{N'}_{0}$ is a right $\mathcal{M}$-approximation. 
Thus we have a morphism $b\in\mathcal{C}(X, M^{N'}_{0})$ with $a=f^{N'}_{0}b$.
By Lemma \ref{lem-sconfls}(2), $M^{N'}_{0} \in \mathcal{M}_{X}$. 
Since $\alpha$ is a left $\mathcal{M}_{X}$-approximation of $X$, we have $c\in\mathcal{C}(M', M^{N'}_{0})$ with $b=c\alpha$. 
By $a=f^{N'}_{0}b=f^{N'}_{0}c\alpha$, the map $\mathcal{C}(\alpha,N'): \mathcal{C}(M', N')\rightarrow\mathcal{C}(X,N')$ is surjective. 
Hence we have $\mathbb{E}(C,N')=0$. This completes the proof.
\end{proof}

Define a silting quiver $Q(\operatorname{\mathsf{silt}}\mathcal{C})=(Q_{0},Q_{1})$ as follows: 
\begin{align}
&Q_{0}:=\operatorname{\mathsf{silt}}\mathcal{C}, \notag\\
&Q_{1}:=\{ \mathcal{M}\rightarrow \mathcal{N}\mid \text{$\mathcal{N}$ is an irreducible left mutation of $\mathcal{M}$}\}. \notag
\end{align}
We compare the silting quiver $Q(\operatorname{\mathsf{silt}}\mathcal{C})$ and the Hasse quiver of $(\operatorname{\mathsf{silt}}\mathcal{C}, \geq)$.

\begin{theorem}\label{thm-comb}
Let $\mathcal{C}$ be a Hom-finite Krull--Schmidt extriangulated category satisfying the condition \textnormal{(F)}.  
Let $\mathcal{M},\mathcal{N}$ be silting subcategories of $\mathcal{C}$. 
Then the following statements are equivalent. 
\begin{itemize}
\item[(1)] $\mathcal{N}$ is an irreducible left mutation of $\mathcal{M}$.
\item[(2)] $\mathcal{M}$ is an irreducible right mutation of $\mathcal{N}$.
\item[(3)] $\mathcal{M}>\mathcal{N}$ and there is no silting subcategory $\mathcal{L}$ such that $\mathcal{M}>\mathcal{L}>\mathcal{N}$. 
\end{itemize}
In particular, $Q(\operatorname{\mathsf{silt}}\mathcal{C})$ coincides with the Hasse quiver of $(\operatorname{\mathsf{silt}}\mathcal{C},\geq)$.
\end{theorem}

\begin{proof}
(3)$\Rightarrow$(1): By Proposition \ref{prop-mut}, we obtain an irreducible left mutation $\mathcal{L}$ of $\mathcal{M}$ such that $\mathcal{M}>\mathcal{L}\geq\mathcal{N}$. Hence we have $\mathcal{L}=\mathcal{N}$.

(1)$\Rightarrow$(3): Let $\mathcal{N}:=\mu_{X}(\mathcal{M})$ for some indecomposable object $X\in\mathcal{M}$.
Then we obtain $\mathcal{M}\cap\mathcal{N}=\mathcal{M}_{X}$.
Suppose that there exists a silting subcategory $\mathcal{L}$ such that $\mathcal{M}>\mathcal{L}>\mathcal{N}$.
By Proposition \ref{prop-mut}, there exists an indecomposable object $Y\in\mathcal{M}$ such that $\mathcal{K}:=\mu_{Y}(\mathcal{M})$ and $\mathcal{M}>\mathcal{K}\geq\mathcal{L}>\mathcal{N}$. 
It follows from Proposition \ref{thm-AT22} that 
\begin{align}
\mathcal{M}_{X}=\mathcal{M}\cap\mathcal{N}\subseteq\mathcal{M}^{\vee}\cap\mathcal{N}^{\wedge}\subseteq\mathcal{K}^{\vee}\cap\mathcal{K}^{\wedge}=\mathcal{K}.\notag
\end{align}
Thus we have $X\cong Y$, and hence $\mathcal{N}=\mathcal{K}$, a contradiction.

Similarly, we can show (2)$\Leftrightarrow$(3). This completes the proof.
\end{proof}

As an application, we have the following results.

\begin{corollary}\label{cor-finsilt}
Let $\mathcal{M}, \mathcal{N}$ be silting subcategories satisfying $\mathcal{M}>\mathcal{N}$.
If there exist only finitely many silting subcategories $\mathcal{L}$ with $\mathcal{M}\geq \mathcal{L}\geq\mathcal{N}$, then $\mathcal{N}$ is obtained from $\mathcal{M}$ by iterated irreducible left mutation.
\end{corollary}

\begin{proof}
By Proposition \ref{prop-mut}, there exists an irreducible left mutation $\mathcal{L}_{1}$ of $\mathcal{M}$ such that $\mathcal{M}>\mathcal{L}_{1}\geq\mathcal{N}$.
If $\mathcal{L}_{1}=\mathcal{N}$, then the assertion holds.
We assume $\mathcal{L}_{1}\neq \mathcal{N}$.
By Proposition \ref{prop-mut}, we have an irreducible left mutation $\mathcal{L}_{2}$ of $\mathcal{L}_{1}$ satisfying $\mathcal{L}_{1}>\mathcal{L}_{2}\geq \mathcal{N}$.
Repeating this process, we obtain a sequence $\mathcal{M}>\mathcal{L}_{1}>\mathcal{L}_{2}>\cdots (\geq \mathcal{N})$ of irreducible left mutations. 
By our assumption, there exists an integer $l\geq 1$ such that $\mathcal{L}_{l}=\mathcal{N}$. 
This completes the proof.
\end{proof}

We give an analog of \cite[Corollary 2.2]{HU05}.

\begin{corollary}\label{cor-mutend}
Assume that $\operatorname{\mathsf{silt}}\mathcal{C}$ has a greatest element $\mathcal{A}$.
If $Q(\operatorname{\mathsf{silt}}\mathcal{C})$ has a finite component $\mathfrak{C}$, then $\mathfrak{C}$ is connected and $\mathfrak{C}=Q(\operatorname{\mathsf{silt}}\mathcal{C})$.
\end{corollary}

\begin{proof}
Let $\mathcal{M}\in \mathfrak{C}$.
We show $\mathcal{A}\in\mathfrak{C}$.
We may assume $\mathcal{M}\neq\mathcal{A}$.
Since $\mathcal{A}$ is the greatest element, we obtain $\mathcal{A}> \mathcal{M}$.
By a dual statement of Proposition \ref{prop-mut}, there exists an irreducible right mutation $\mathcal{M}_{1}$ of $\mathcal{M}$ such that $\mathcal{A}\geq \mathcal{M}_{1}>\mathcal{M}$. 
If $\mathcal{A}=\mathcal{M}_{1}$, then there is nothing to prove.
Thus we assume $\mathcal{A}\neq\mathcal{M}_{1}$.
Repeating this process, we have a sequence $(\mathcal{A}\geq)\cdots >\mathcal{M}_{2}>\mathcal{M}_{1}>\mathcal{M}_{0}:=\mathcal{M}$ of irreducible right mutations in $\mathfrak{C}$. 
Since $\mathfrak{C}$ is finite, there exists an integer $l\geq 1$ such that $\mathcal{A}=\mathcal{M}_{l}$.
Thus we obtain $\mathcal{A}\in \mathfrak{C}$.
Let $\mathcal{N}$ be an arbitrary silting subcategory satisfying $\mathcal{N}\neq\mathcal{A}$.
Since $\mathcal{A}$ is the greatest element, we have $\mathcal{A}>\mathcal{N}$.
Thus it follows from Proposition \ref{prop-mut} that there exists a sequence $\mathcal{A}=:\mathcal{A}_{0}> \mathcal{A}_{1}>\mathcal{A}_{2}>\cdots (\geq\mathcal{N})$ of irreducible left mutations in $\mathfrak{C}$. 
Since $\mathfrak{C}$ is finite, we have $\mathcal{N}=\mathcal{A}_{k}$ for some $k\geq 1$. 
Thus $\mathcal{N}$ belongs to $\mathfrak{C}$.
This completes the proof.
\end{proof}

\subsection{Bongartz completion}

Assume that $R$ is a field. Let $\mathcal{C}$ be an $R$-linear extriangulated category.
The aim of this subsection is to give sufficient conditions for presilting subcategories to be partial silting subcategories. 
For a silting subcategory $\mathcal{M}$ and a presilting subcategory $\mathcal{N}$, a subset $\nabla_{\mathcal{M}}(\mathcal{N})$ of $\operatorname{\mathsf{silt}}\mathcal{C}$ is defined as $\nabla_{\mathcal{M}}(\mathcal{N}):=\{\mathcal{L} \in \operatorname{\mathsf{silt}}\mathcal{C}\mid \mathcal{L}\geq \mathcal{N}, \mathcal{L} \subseteq \mathcal{M}^{\wedge}_{1}\}$. 

The following theorem is one of main results of this section, which is an extriangulated version of \cite[Proposition 2.14]{AM17}. 

\begin{theorem}\label{thm-BC1}
Let $\mathcal{C}$ be a Hom-finite Krull--Schmidt extriangulated category satisfying the condition \textnormal{(F)}. 
Let $\mathcal{M}$ be a silting subcategory and let $\mathcal{N}$ be a presilting subcategory with $\mathcal{N}\subseteq\mathcal{M}^{\wedge}_{l}$ for some $l\geq0$. 
Assume that for each silting subcategory $\mathcal{L}$ with $\mathcal{M}\geq \mathcal{L} \geq \mathcal{N}$, there exists a minimal element in $\nabla_{\mathcal{L}}(\mathcal{N})$. 
Then $\mathcal{N}$ is a partial silting subcategory. 
\end{theorem}

To show Theorem \ref{thm-BC1}, we need the following lemma. 

\begin{lemma}\label{lem-316}
Let $\mathcal{M}$ be a silting subcategory of $\mathcal{C}$ and let $\mathcal{N}$ be a presilting subcategory of $\mathcal{C}$ with $\mathcal{N}\subseteq \mathcal{M}^{\wedge}_{l}$ for some $l\geq 1$. 
Assume that $\nabla_{\mathcal{M}}(\mathcal{N})$ has a minimal element $\mathcal{Q}$.
Then $\mathcal{N}\subseteq \mathcal{Q}^{\wedge}_{l-1}$.
\end{lemma}

\begin{proof}
By Lemmas \ref{lem-wedge_perp}(2) and \ref{lem-order_wedge}(1), we obtain 
\begin{align}
\mathcal{N}\subseteq \mathcal{Q}^{\wedge}\cap {}^{\perp_{>l}}\mathcal{Q}=\mathcal{Q}^{\wedge}_{l}.\notag
\end{align}
Suppose to the contrary that $\mathcal{N}\not\subseteq \mathcal{Q}^{\wedge}_{l-1}$.
Then there exists an object $N\in\mathcal{N}$ such that it is in $\mathcal{Q}^{\wedge}_{l}$ but not in $\mathcal{Q}^{\wedge}_{l-1}$.
By Lemma \ref{lem-sconfls}(1), there exists an $\mathfrak{s}$-conflation 
\begin{align}\label{seq-2015}
Q^{N}_{l}\xrightarrow{g_{l}^{N}} Q^{N}_{l-1}\xrightarrow{f_{l-1}^{N}} N_{l-1}\dashrightarrow
\end{align}
such that $Q_{l}^{N}\neq0$ and $g_{l}^{N}$ is in the Jacobson radical of $\mathcal{C}$.
Let $X$ be an indecomposable direct summand of $Q^{N}_{l}$.
By Proposition \ref{prop-mut}, there exists an $\mathfrak{s}$-conflation
\begin{align}\label{seq-2207}
X\xrightarrow{a}Q'\rightarrow Y\dashrightarrow 
\end{align}
such that $a$ is a left $\mathcal{Q}_{X}$-approximation of $X$ and $\mathcal{R}:=\mu_{X}(\mathcal{Q})=\operatorname{\mathsf{add}}(\mathcal{Q}_{X}\cup\{Y\})$ is a silting subcategory of $\mathcal{C}$ with $\mathcal{Q}>\mathcal{R}\geq \mathcal{N}$.
We show $\mathcal{R}\in\nabla_{\mathcal{M}}(\mathcal{N})$, or equivalently $\mathcal{R}\subseteq \mathcal{M}^{\wedge}_{1}$.
By Lemma \ref{lem-wedge_perp}(2), it is enough to prove $\mathcal{R}\subseteq {}^{\perp_{>1}}\mathcal{M}$.
Let $i\geq 2$.
Applying $\mathcal{C}(-,\mathcal{M})$ to \eqref{seq-2207} induces an exact sequence
\begin{align}
\mathbb{E}^{i-1}(Q',\mathcal{M})\xrightarrow{\mathbb{E}^{i-1}(a,\mathcal{M})}\mathbb{E}^{i-1}(X,\mathcal{M})\rightarrow\mathbb{E}^{i}(Y,\mathcal{M})\rightarrow \mathbb{E}^{i}(Q',\mathcal{M})=0,\notag
\end{align}
where the last equality follows from Lemma \ref{lem-wedge_perp}(2).
If $i\geq 3$, then $\mathbb{E}^{i-1}(X,\mathcal{M})$ also vanishes by Lemma \ref{lem-wedge_perp}(2).
Thus we obtain $\mathbb{E}^{i}(Y,\mathcal{M})=0$ for all $i\geq 3$.
Assume $i=2$. 
We show that $\mathbb{E}(a,M)$ is surjective for each $M\in\mathcal{M}$.
Let $\eta\in \mathbb{E}(X,M)$.
By Lemma \ref{lem-order_wedge}(3), $\mathcal{Q}\subseteq\mathcal{M}^{\wedge}_{1}$ implies $\mathcal{M}\subseteq\mathcal{Q}^{\vee}_{1}$.
Thus there exists an $\mathfrak{s}$-conflation 
\begin{align}\label{seq-2232}
M\xrightarrow{f}Q_{0}\xrightarrow{g}Q_{1}\overset{\delta}{\dashrightarrow}
\end{align}
such that $Q_{0}, Q_{1}\in\mathcal{Q}$ and $g$ is in the Jacobson radical of $\mathcal{C}$.
Applying $\mathcal{C}(X,-)$ to \eqref{seq-2232} induces an exact sequence
\begin{align}
\mathcal{C}(X,Q_{1}) \rightarrow \mathbb{E}(X,M)\rightarrow \mathbb{E}(X,Q_{0})=0, \notag
\end{align}
where the last equality follows from the assumption that $\mathcal{Q}$ is silting.
Thus we have a morphism $b\in \mathcal{C}(X,Q_{1})$ such that $b^{\ast}\delta=\eta$.
We claim $Q_{1}\in \mathcal{Q}_{X}$. By $X\in \operatorname{\mathsf{add}}Q^{N}_{l}$, it is enough to show $\operatorname{\mathsf{add}}Q_{1}\cap \operatorname{\mathsf{add}}Q^{N}_{l}=\{0\}$.
Let $\alpha\in \mathcal{C}(Q^{N}_{l},Q_{1})$.
Applying $\mathcal{C}(-,M)$ to \eqref{seq-2015}, we have an exact sequence
\begin{align}
\mathbb{E}(Q^{N}_{l-1},M)\xrightarrow{\mathbb{E}(g^{N}_{l},M)}\mathbb{E}(Q^{N}_{l},M)\rightarrow\mathbb{E}^{2}(N_{l-1},M).\notag
\end{align}
By Lemma \ref{lem-wedge_perp}(2), we have isomorphisms
\begin{align}
\mathbb{E}^{2}(N_{l-1},M)\cong \cdots \cong \mathbb{E}^{l+1}(N,M)=0. \notag 
\end{align}
Thus there exists $\theta \in \mathbb{E}(Q^{N}_{l-1}, M)$ such that $\alpha^{\ast}\delta=\mathbb{E}(g^{N}_{l},M)(\theta)=(g^{N}_{l})^{\ast}\theta$.
Applying $\mathcal{C}(Q^{N}_{l-1},-)$ to \eqref{seq-2232} induces an exact sequence
\begin{align}
\mathcal{C}(Q^{N}_{l-1},Q_{1})\rightarrow \mathbb{E}(Q^{N}_{l-1},M)\rightarrow\mathbb{E}(Q^{N}_{l-1},Q_{0})=0, \notag
\end{align}
where the last equality follows from the assumption that $\mathcal{Q}$ is silting.
Therefore we have $\beta\in \mathcal{C}(Q^{N}_{l-1},Q_{1})$ with $\beta^{\ast}\delta=\theta$.
By $(\alpha-\beta g^{N}_{l})^{\ast}\delta=\alpha^{\ast}\delta-(g^{N}_{l})^{\ast}(\beta^{\ast}\delta)=\alpha^{\ast}\delta-(g^{N}_{l})^{\ast}\theta=0$, there exists $\beta'\in \mathcal{C}(Q^{N}_{l},Q_{0})$ such that $\alpha-\beta g^{N}_{l}=g\beta'$.
Since $g$ and $g^{N}_{l}$ are in the Jacobson radical of $\mathcal{C}$, we have $\alpha=\beta g^{N}_{l}+g\beta'\in \mathrm{rad}_{\mathcal{C}}(Q^{N}_{l},Q_{1})$, and hence $\mathcal{C}(Q^{N}_{l},Q_{1})=\mathrm{rad}_{\mathcal{C}}(Q^{N}_{l},Q_{1})$.
This implies $\operatorname{\mathsf{add}}Q^{N}_{l}\cap \operatorname{\mathsf{add}}Q_{1}=\{ 0\}$.
Hence $Q_{1}\in \mathcal{Q}_{X}$.
Since $a$ is a left $\mathcal{Q}_{X}$-approximation, there exists $c\in \mathcal{C}(Q',Q_{1})$ such that $b=ca$.
Thus we have $c^{\ast}\delta\in \mathbb{E}(Q',M)$ and $\mathbb{E}(a,M)(c^{\ast}\delta)=a^{\ast}(c^{\ast}\delta)=(ca)^{\ast}\delta=b^{\ast}\delta=\eta$. This implies that $\mathbb{E}(a,M)$ is surjective.
Hence $\mathcal{R}\in \nabla_{\mathcal{M}}(\mathcal{N})$.
This contradicts that $\mathcal{Q}$ is a minimal element in $\nabla_\mathcal{M}(\mathcal{N})$.
Therefore we obtain $\mathcal{N}\subseteq\mathcal{Q}^{\wedge}_{l-1}$.
\end{proof}

We are ready to prove Theorem \ref{thm-BC1}. 

\begin{proof}[Proof of Theorem \ref{thm-BC1}]
If $l=0$, then this is clear. 
Assume $l\neq 0$.  
Let $\mathcal{M}_{1}$ be a minimal element in $\nabla_{\mathcal{M}}(\mathcal{N})$. 
By Lemma \ref{lem-316}, we have $\mathcal{N}\subseteq(\mathcal{M}_{1})_{l-1}^{\wedge}$. 
By repeating this process, we obtain silting subcategories $\mathcal{M}\geq \mathcal{M}_{1}\geq \cdots \geq \mathcal{M}_{l}$ with $\mathcal{N}\subseteq(\mathcal{M}_{i})^{\wedge}_{l-i}$.
Hence $\mathcal{N}\subseteq(\mathcal{M}_{l})^{\wedge}_{0}=\mathcal{M}_{l}$. 
This completes the proof.
\end{proof}

By Theorem \ref{thm-BC1}, we have the following result. 

\begin{corollary}\label{cor-BC1}
Let $\mathcal{M}$ be a silting subcategory of $\mathcal{C}$ and let $\mathcal{N}$ be a presilting subcategory of $\mathcal{C}$ with $\mathcal{M}> \mathcal{N}$. 
If there exist only finitely many silting subcategories $\mathcal{L}$ with $\mathcal{M}\geq \mathcal{L} \geq \mathcal{N}$, then $\mathcal{N}$ is a partial silting subcategory. 
\end{corollary} 

\begin{proof}
For each silting subcategory $\mathcal{L}$ satisfying $\mathcal{M}\geq \mathcal{L}\geq\mathcal{N}$, the set $\nabla_{\mathcal{L}}(\mathcal{N})$ is finite by our assumption. Hence $\nabla_{\mathcal{L}}(\mathcal{N})$ has a minimal element. By Theorem \ref{thm-BC1}, we obtain the assertion. 
\end{proof}

Taking universal coextensions, we provide another sufficient condition for presilting objects to be partial silting objects.
Recall the definition of universal coextensions.
Let $M,N$ be objects in $\mathcal{C}$.
Assume that the vector space $\mathbb{E}(N,M)\neq 0$ is finite dimensional and let $d:=\dim\mathbb{E}(N,M)$. 
Take a basis $\delta_{1},\delta_{2}, \ldots ,\delta_{d}$ of $\mathbb{E}(N,M)$.
Then we have an $\mathfrak{s}$-conflation $M\xrightarrow{f_{i}}F_{i}\xrightarrow{g_{i}}N\overset{\delta_{i}}{\dashrightarrow}$ for each $i \in [1, d]$.
It follows from \cite[Corollary 3.16]{NP19} that $\small\left[\begin{smallmatrix}\mathrm{id}&\cdots&\mathrm{id}\end{smallmatrix}\right]: M^{\oplus d}\rightarrow M$ is an $\mathfrak{s}$-deflation.
Thus, by (ET4), we have a commutative diagram of $\mathfrak{s}$-conflations
\begin{align}
\xymatrix{
M^{\oplus d}\ar[r]^-{\oplus f_{i}}\ar[d]_-{\small\left[\begin{smallmatrix}\mathrm{id}&\cdots&\mathrm{id}\end{smallmatrix}\right]}&\oplus_{i=1}^{d}F_{i}\ar[r]^-{\oplus g_{i}}\ar[d]&N^{\oplus d}\ar@{=}[d]\ar@{-->}[r]^{\oplus \delta_{i}}&\;\\
M\ar[r]^-{f}&F\ar[r]^-{g}&N^{\oplus d}\ar@{-->}[r]^{\delta}&.
}\notag
\end{align}
Then we call the $\mathfrak{s}$-conflation $M\xrightarrow{f} F\xrightarrow{g} N^{\oplus d}\overset{\delta}\dashrightarrow$ a \emph{universal coextension of $M$ by $N$}.
By a construction of universal coextensions, $(\delta_{\sharp})_{N}: \mathcal{C}(N,N^{\oplus d})\to \mathbb{E}(N,M)$ is surjective.
The following theorem is an analog of Bongartz's lemma (\cite{B81}). 

\begin{theorem}\label{thm-315}
Let $M$ be a silting object of $\mathcal{C}$ and $N$ a presilting object of $\mathcal{C}$.
Assume that the vector space $\mathbb{E}(N,M)$ is finite dimensional. 
If $N\in (\operatorname{\mathsf{add}}M)^{\wedge}_{1}$, then $N$ is partial silting.
\end{theorem}

\begin{proof}
Let $d:=\dim\mathbb{E}(N,M)$. 
If $d=0$, then $N\in(\operatorname{\mathsf{add}}M)^{\wedge}_{1}$ induces a split $\mathfrak{s}$-conflation $M_{1}\rightarrow M_{0}\rightarrow N\dashrightarrow$ with $M_{1},M_{0}\in \operatorname{\mathsf{add}}M$. Hence $N\in \operatorname{\mathsf{add}}M$. 
Assume $d\neq 0$.
Taking a universal coextension, we have an $\mathfrak{s}$-conflation 
\begin{align}\label{seq-univext2212}
M\rightarrow F\rightarrow N^{\oplus d}\dashrightarrow.
\end{align}
Then we prove that $T:=F\oplus N$ is a silting object of $\mathcal{C}$.
Since $\operatorname{\mathsf{thick}}T$ is closed under cocones, $F,N\in \operatorname{\mathsf{thick}}T$ implies $M\in \operatorname{\mathsf{thick}}T$.
Hence $\mathcal{C}=\operatorname{\mathsf{thick}}M\subseteq \operatorname{\mathsf{thick}}T$.
We show $\mathbb{E}^{k}(T,T)=0$ for all $k\geq 1$.

(i) Applying $\mathcal{C}(-,N)$ to \eqref{seq-univext2212}, we have an exact sequence
\begin{align}
0=\mathbb{E}^{k}(N^{\oplus d},N)\rightarrow \mathbb{E}^{k}(F,N)\rightarrow \mathbb{E}^{k}(M,N). \notag
\end{align}
Since it follows from Lemma \ref{lem-wedge_perp}(1) that $\mathbb{E}^{k}(M,N)=0$, we have $\mathbb{E}^{k}(F,N)=0$ for all $k\geq 1$.

(ii) Applying $\mathcal{C}(N,-)$ to \eqref{seq-univext2212} induces an exact sequence
\begin{align}
\mathbb{E}^{k-1}(N,N^{\oplus d})\rightarrow \mathbb{E}^{k}(N,M)\rightarrow \mathbb{E}^{k}(N,F)\rightarrow \mathbb{E}^{k}(N,N^{\oplus d})=0.\notag
\end{align}
By Lemma \ref{lem-wedge_perp}(2), we have $\mathbb{E}^{k}(N,M)=0$ for all $k\geq2$.
Moreover, by a property of the universal coextension, the map $\mathcal{C}(N,N^{\oplus d})\rightarrow \mathbb{E}(N,M)$ is a surjective. Hence $\mathbb{E}^{k}(N,F)=0$ for all $k\geq1$.

(iii) By applying $\mathcal{C}(M,-)$ to \eqref{seq-univext2212}, we obtain an exact sequence
\begin{align}
0=\mathbb{E}^{k}(M,M)\rightarrow \mathbb{E}^{k}(M,F)\rightarrow \mathbb{E}^{k}(M,N^{\oplus d}). \notag
\end{align}
By Lemma \ref{lem-wedge_perp}(1), $\mathbb{E}^{k}(M,N^{\oplus d})=0$, and hence $\mathbb{E}^{k}(M,F)=0$ for all $k\geq 1$.
Moreover, applying $\mathcal{C}(-,F)$ to \eqref{seq-univext2212} induces an exact sequence
\begin{align}
\mathbb{E}^{k}(N^{\oplus d}, F) \rightarrow \mathbb{E}^{k}(F,F) \rightarrow \mathbb{E}^{k}(M,F)=0.\notag
\end{align}
Since the left-hand side vanishes by (ii), we obtain $\mathbb{E}^{k}(F,F)=0$ for all $k\geq1$.

By (i), (ii) and (iii), $T$ is a silting object in $\mathcal{C}$.
Hence $N$ is partial silting.
\end{proof}

\subsection{Recover Auslander--Reiten's result}

Our aim of this subsection is to prove the following theorem, which is an analog of Auslander--Reiten's result (\cite[Theorem 5.5]{AR91}).

\begin{theorem}\label{thm-AR2}
Let $\mathcal{C}$ be a weakly idempotent complete extriangulated category which has enough projective objects and enough injective objects. 
Assume that each silting subcategory of $\mathcal{P}^{\infty}$ with finite projective dimension is contravariantly finite in $\mathcal{C}$. 
Then there exist mutually inverse bijections
\begin{align}
\xymatrix{
\{\mathcal{M}\in\operatorname{\mathsf{silt}}\mathcal{P}^{\infty}\mid \mathrm{pd}\mathcal{M}<\infty\}\ar@<0.5ex>[d]^-{\varphi} \\
\{\mathcal{X}:\textnormal{contravariantly finite resolving subcategory of }\mathcal{C}\mid\mathrm{pd} \mathcal{X}<\infty\}\ar@<0.5ex>[u]^-{\psi}, 
}\notag\end{align}
where $\varphi(\mathcal{M}):=\mathcal{M}^{\vee}$ and $\psi(\mathcal{X}):=\mathcal{X}\cap \mathcal{X}^{\perp}$.
\end{theorem}

To show Theorem \ref{thm-AR2}, the following theorem plays an important role.

\begin{theorem}\label{thm-AR}
Let $\mathcal{C}$ be a weakly idempotent complete extriangulated category which has enough projective objects and enough injective objects.  
Assume that each silting subcategory of $\mathcal{P}^{\infty}$ with finite projective dimension is contravariantly finite in $\mathcal{C}$. 
Then there exist mutually inverse bijections
\begin{align}
\xymatrix{
\{\mathcal{M}\in\operatorname{\mathsf{silt}}\mathcal{P}^{\infty}\mid \mathrm{pd}\mathcal{M}<\infty\}\ar@<0.5ex>[r]^-{\Phi} &\{(\mathcal{X},\mathcal{Y})\in\operatorname{\mathsf{hcotors}}\mathcal{C}\mid\mathrm{pd}\mathcal{X}<\infty\}\ar@<0.5ex>[l]^-{\Psi}, 
}\notag\end{align}
where $\Phi(\mathcal{M}):=(\mathcal{M}^{\vee},\mathcal{M}^{\perp})$ and $\Psi(\mathcal{X},\mathcal{Y}):=\mathcal{X}\cap\mathcal{Y}$.
\end{theorem}

Before showing Theorem \ref{thm-AR}, we give a proof of Theorem \ref{thm-AR2}.

\begin{proof}[Proof of Theorem \ref{thm-AR2}]
By \cite[Proposition 5.15]{AT22}, there exist mutually inverse bijections \begin{align}
\xymatrix{
\{(\mathcal{X},\mathcal{Y})\in\operatorname{\mathsf{hcotors}}\mathcal{C}\mid\mathrm{pd}\mathcal{X}<\infty\}\ar@<0.5ex>[d]^-{\varphi'} \\
\{\mathcal{X}:\textnormal{contravariantly finite resolving subcategory of }\mathcal{C}\mid \mathrm{pd}\mathcal{X}<\infty\},\ar@<0.5ex>[u]^-{\psi'}
}\notag\end{align}
where $\varphi'(\mathcal{X},\mathcal{Y}):=\mathcal{X}$ and $\psi'(\mathcal{X}):=(\mathcal{X},\mathcal{X}^{\perp})$.
Thus the assertion follows from Theorem \ref{thm-AR}.
\end{proof}

In the following, we prove Theorem \ref{thm-AR}. 
First, we show that $\Psi$ is well-defined. 

\begin{proposition}\label{prop-Phi}
Assume that $\mathcal{C}$ has enough projective objects. 
If $(\mathcal{X},\mathcal{Y})$ is a hereditary cotorsion pair in $\mathcal{C}$ with $\mathrm{pd}
\mathcal{X}<\infty$, then $\mathcal{X}\cap\mathcal{Y}$ is a silting subcategory of $\mathcal{P}^{\infty}$ with $\mathrm{pd} (\mathcal{X}\cap\mathcal{Y})<\infty$.
\end{proposition}
 
\begin{proof}
Let $(\mathcal{X},\mathcal{Y})\in \operatorname{\mathsf{hcotors}}\mathcal{C}$ with $\mathrm{pd}\mathcal{X}<\infty$. 
Then $\mathcal{X}\subseteq\mathcal{P}^{\infty}$. 
By Proposition \ref{prop-AT}, it is enough to show that $(\mathcal{X},\mathcal{Y}\cap\mathcal{P}^{\infty})$ is a hereditary cotorsion pair in $\mathcal{P}^{\infty}$. 
The subcategories $\mathcal{X}$ and $\mathcal{Y}\cap\mathcal{P}^{\infty}$ are clearly closed under direct summands.
By Proposition \ref{prop-propertypfin}, $\mathcal{P}^{\infty}$ is a thick subcategory and a resolving subcategory.
Thus it follows from Lemma \ref{lem-thickE} that 
\begin{align}
\mathbb{E}_{\mathcal{P}^{\infty}}^{i}(X,Y)\cong \mathbb{E}^{i}(X,Y)=0\notag
\end{align}
for each $X\in\mathcal{X}, Y\in\mathcal{Y}\cap\mathcal{P}^{\infty}$ and $i \geq 1$. 
Since $\mathcal{P}^{\infty}$ is a thick subcategory of $\mathcal{C}$, we have $\mathcal{P}^{\infty}\supseteq \operatorname{\mathsf{cone}}(\mathcal{Y}\cap\mathcal{P}^{\infty},\mathcal{X})$ and $\mathcal{P}^{\infty}\supseteq\operatorname{\mathsf{cocone}}(\mathcal{Y}\cap\mathcal{P}^{\infty},\mathcal{X})$.
We show the converse inclusions. 
Let $M\in\mathcal{P}^{\infty}$. 
By $(\mathcal{X},\mathcal{Y})\in\operatorname{\mathsf{hcotors}}\mathcal{C}$, there exist $\mathfrak{s}$-conflations $Y_{M}\rightarrow X_{M}\rightarrow M \dashrightarrow$ and $M\rightarrow Y^{M}\rightarrow X^{M}\dashrightarrow$ such that $Y_{M}, Y^{M}\in\mathcal{Y}$ and $X_{M}, X^{M}\in\mathcal{X}$. 
Since $\mathcal{X}\subseteq \mathcal{P}^{\infty}$ and $\mathcal{P}^{\infty}$ is a resolving subcategory, we have $Y_{M}, Y^{M}\in\mathcal{Y}\cap\mathcal{P}^{\infty}$.
Thus $\mathcal{X}\cap\mathcal{Y}$ is a silting subcategory of $\mathcal{P}^{\infty}$ with $\mathrm{pd} (\mathcal{X}\cap\mathcal{Y})<\infty$. 
\end{proof}

Next, we show that $\Phi$ is well-defined. 

\begin{proposition}\label{prop-Psi}
Assume that $\mathcal{C}$ is weakly idempotent complete and has enough projective objects and enough injective objects. 
Let $\mathcal{M}$ be a silting subcategory of $\mathcal{P}^{\infty}$ with $\mathrm{pd}\mathcal{M}<\infty$.
Assume that $\mathcal{M}$ is a contravariantly finite subcategory of $\mathcal{C}$.
Then $(\mathcal{M}^{\vee},\mathcal{M}^{\perp})$ is a hereditary cotorsion pair in $\mathcal{C}$ with $\mathrm{pd}\mathcal{M}^{\vee}<\infty$.
\end{proposition}

\begin{proof}
By Lemma \ref{lem-pdim_vee}, $\mathrm{pd}\mathcal{M}^{\vee}=\mathrm{pd}\mathcal{M}<\infty$. 
We show $(\mathcal{M}^{\vee},\mathcal{M}^{\perp})\in \operatorname{\mathsf{hcotors}}\mathcal{C}$.
It follows from Lemma \ref{lem-psilt_wedge}(1) that $\mathcal{M}^{\vee}$ and $\mathcal{M}^{\perp}$ are closed under direct summands.
By Lemma \ref{lem-wedge_perp}(1), we obtain $\mathbb{E}^{i}(\mathcal{M}^{\vee}, \mathcal{M}^{\perp})=0$ for each $i\geq 1$.
To show $\mathcal{C}=\operatorname{\mathsf{cocone}}(\mathcal{M}^{\perp},\mathcal{M}^{\vee})=\operatorname{\mathsf{cone}}(\mathcal{M}^{\perp},\mathcal{M}^{\vee})$, we claim that (i) $\mathcal{C}=(\mathcal{M}^{\perp})^{\vee}_{d}$, where $d:=\mathrm{pd}\mathcal{M}$, and (ii) $\mathcal{M}^{\perp}=\operatorname{\mathsf{cone}}(\mathcal{M}^{\perp},\mathcal{M})$. 
First, we show (i). 
Let $X\in \mathcal{C}$.
Since $\mathcal{C}$ has enough injective objects, we have an $\mathfrak{s}$-conflation $X\rightarrow I^{0}\rightarrow C^{1}\dashrightarrow$ with $I^{0}\in\operatorname{\mathsf{inj}}\mathcal{C}$. 
Inductively, there exists an $\mathfrak{s}$-conflation $C^{i}\rightarrow I^{i}\rightarrow C^{i+1}\dashrightarrow$ such that $I^{i}\in\operatorname{\mathsf{inj}}\mathcal{C}$. 
Applying $\mathcal{C}(\mathcal{M},-)$ to the $\mathfrak{s}$-conflations above induces isomorphisms
\begin{align}
\mathbb{E}^{j}(\mathcal{M},C^{d})\cong \mathbb{E}^{j+1}(\mathcal{M},C^{d-1})\cong \cdots \cong \mathbb{E}^{j+d}(\mathcal{M}, X)=0 \notag
\end{align}
for each $j\geq 1$, where the last equality follows from Lemma \ref{lem-516}. 
By $I^{i}, C^{d}\in\mathcal{M}^{\perp}$, we have $X\in(\mathcal{M}^{\perp})^{\vee}_{d}$.
Next, we show (ii).
Since $\mathcal{M}^{\perp}$ is closed under cones, we obtain
\begin{align}
\operatorname{\mathsf{cone}}(\mathcal{M}^{\perp},\mathcal{M})\subseteq \operatorname{\mathsf{cone}}(\mathcal{M}^{\perp},\mathcal{M}^{\perp})\subseteq \mathcal{M}^{\perp}.\notag
\end{align} 
We claim the converse inclusion.
Let $X\in\mathcal{M}^{\perp}$.
Since $\mathcal{M}$ is contravariantly finite, there exists a right $\mathcal{M}$-approximation $f: M\rightarrow X$.
We show that $f$ is an $\mathfrak{s}$-deflation.
Since $\mathcal{C}$ has enough projective objects, we obtain $\mathcal{C}=\operatorname{\mathsf{cone}}(\mathcal{C},\operatorname{\mathsf{proj}}\mathcal{C})$, and hence
\begin{align}
\mathcal{M}^{\perp}
&\subseteq \operatorname{\mathsf{cone}}(\mathcal{C},\operatorname{\mathsf{proj}}\mathcal{C})&\notag\\
&\subseteq \operatorname{\mathsf{cone}}(\mathcal{C},\mathcal{M}^{\vee})&\textnormal{by Proposition \ref{prop-siltpfin}(2)}\notag\\
&=\operatorname{\mathsf{cone}}(\mathcal{C},\operatorname{\mathsf{cocone}}(\mathcal{M},\mathcal{M}^{\vee}))&\notag\\
&\subseteq\operatorname{\mathsf{cocone}}(\operatorname{\mathsf{cone}}(\mathcal{C},\mathcal{M}),\mathcal{M}^{\vee})&\textnormal{by (ET4)}.\notag
\end{align}
Thus there exists an $\mathfrak{s}$-conflation $X\rightarrow E\rightarrow C\dashrightarrow$ such that $E\in \operatorname{\mathsf{cone}}(\mathcal{C},\mathcal{M})$ and $C\in \mathcal{M}^{\vee}$. 
By $X\in \mathcal{M}^{\perp}=(\mathcal{M}^{\vee})^{\perp}$ and $C\in\mathcal{M}^{\vee}$, the $\mathfrak{s}$-conflation splits.
Thus we have an $\mathfrak{s}$-deflation $f': M'\to X$ with $M'\in \mathcal{M}$. 
Since $f:M\rightarrow X$ is a right $\mathcal{M}$-approximation, there exists $g:M'\rightarrow X$ such that $fg=f'$.
By our assumption that $\mathcal{C}$ is weakly idempotent complete, $f$ is an $\mathfrak{s}$-deflation.
Thus we have an $\mathfrak{s}$-conflation $X'\rightarrow M\xrightarrow{f}X\dashrightarrow$.
By $M,X\in \mathcal{M}^{\perp}$, we obtain $X'\in\mathcal{M}^{\perp_{>1}}$.
Applying $\mathcal{C}(\mathcal{M},-)$ to the $\mathfrak{s}$-conflation above induces an exact sequence
\begin{align}
\mathcal{C}(\mathcal{M},M)\xrightarrow{\mathcal{C}(\mathcal{M},f)} \mathcal{C}(\mathcal{M},X)\rightarrow \mathbb{E}(\mathcal{M},X')\rightarrow \mathbb{E}(\mathcal{M},M)=0.\notag
\end{align}
Since $f$ is a right $\mathcal{M}$-approximation, we have $\mathbb{E}(\mathcal{M},X')=0$.
Hence $X'\in \mathcal{M}^{\perp}$.
To complete the proof, we show $(\mathcal{M}^{\perp})^{\vee}_{n}\subseteq\operatorname{\mathsf{cocone}}(\mathcal{M}^{\perp},\mathcal{M}^{\vee}_{n-1})\subseteq\operatorname{\mathsf{cone}}(\mathcal{M}^{\perp},\mathcal{M}^{\vee}_{n})$ by induction on $n$. 
Assume $n=0$. By (ii), we have $\mathcal{M}^{\perp}=\operatorname{\mathsf{cocone}}(\mathcal{M}^{\perp},\{ 0\})\subseteq \operatorname{\mathsf{cone}}(\mathcal{M}^{\perp},\mathcal{M})$.
Assume $n\geq 1$. Then we have
\begin{align}
(\mathcal{M}^{\perp})^{\vee}_{n}
&=\operatorname{\mathsf{cocone}}(\mathcal{M}^{\perp},(\mathcal{M}^{\perp})^{\vee}_{n-1})\notag\\
&\subseteq\operatorname{\mathsf{cocone}}(\mathcal{M}^{\perp},\operatorname{\mathsf{cone}}(\mathcal{M}^{\perp}, \mathcal{M}^{\vee}_{n-1}))&\textnormal{by the induction hypothesis}\notag\\
&\subseteq \operatorname{\mathsf{cocone}}(\mathcal{M}^{\perp}\ast\mathcal{M}^{\perp},\mathcal{M}^{\vee}_{n-1})&\textnormal{by \cite[Proposition 3.15]{NP19}}\notag\\
&=\operatorname{\mathsf{cocone}}(\mathcal{M}^{\perp},\mathcal{M}^{\vee}_{n-1})\notag\\
&=\operatorname{\mathsf{cocone}}(\operatorname{\mathsf{cone}}(\mathcal{M}^{\perp}, \mathcal{M}),\mathcal{M}^{\vee}_{n-1})&\textnormal{by (ii)}\notag\\
&\subseteq\operatorname{\mathsf{cone}}(\mathcal{M}^{\perp},\operatorname{\mathsf{cocone}}(\mathcal{M},\mathcal{M}^{\vee}_{n-1}))&\textnormal{by (ET4)$^{\mathrm{op}}$}\notag\\
&=\operatorname{\mathsf{cone}}(\mathcal{M}^{\perp},\mathcal{M}^{\vee}_{n}).\notag
\end{align}
By (i), we have the assertion.
\end{proof}
 
Now, we are ready to prove Theorem \ref{thm-AR}. 
 
\begin{proof}[Proof of Theorem \ref{thm-AR}]
By Propositions \ref{prop-Phi} and \ref{prop-Psi}, it is enough to show $\Psi\Phi=1$ and $\Phi\Psi=1$. 
Let $\mathcal{M}\in\operatorname{\mathsf{silt}}\mathcal{P}^{\infty}$ with $\mathrm{pd}\mathcal{M}<\infty$.
By Lemma \ref{lem-wedge_perp}(3), we obtain $\Psi\Phi(\mathcal{M})=\mathcal{M}^{\vee}\cap\mathcal{M}^{\perp}=\mathcal{M}$.
Hence $\Psi\Phi=1$.
To show $\Phi\Psi=1$, we claim $\mathcal{X}=(\mathcal{X}\cap\mathcal{Y})^{\vee}$ for each $(\mathcal{X},\mathcal{Y})\in\operatorname{\mathsf{hcotors}}\mathcal{C}$ with $\mathrm{pd}\mathcal{X}<\infty$. 
By Proposition \ref{prop-Phi}, $\mathcal{X}\cap\mathcal{Y}$ is a silting subcategory of $\mathcal{P}^{\infty}$.
Thus it follows from Lemmas \ref{lem-psilt_wedge}(5) and \ref{lem_cotors} that 
\begin{align}
\mathcal{X}\subseteq \mathcal{P}^{\infty}\cap{}^{\perp}\mathcal{Y}\subseteq \operatorname{\mathsf{thick}}(\mathcal{X}\cap\mathcal{Y})\cap {}^{\perp}(\mathcal{X}\cap\mathcal{Y})=(\mathcal{X}\cap\mathcal{Y})^{\vee}\subseteq\mathcal{X}^{\vee}=\mathcal{X}.\notag
\end{align}
This completes the proof.
\end{proof}

As an application, we obtain a silting version of \cite[Theorem 5.1]{CP}. 

\begin{corollary}\label{cor-CP}
Let $\mathcal{C}$ be an extriangulated category which has enough projective objects and let $\mathcal{M}$ be a silting subcategory of $\mathcal{P}^{\infty}$.  
Then the following statements are equivalent. 
\begin{enumerate}[\upshape(1)]
\item $\mathcal{C}=\mathcal{P}^{\infty}$. 
\item $\operatorname{\mathsf{thick}}\mathcal{M}=\mathcal{C}$. 
\item The pairs $(\mathcal{M}^{\vee},\mathcal{M}^{\perp})$ and $({}^{\perp}\mathcal{M},\mathcal{M}^{\wedge})$ are hereditary cotorsion paris in $\mathcal{C}$, and $(\mathcal{M}^{\vee},\mathcal{M}^{\perp})$ and $({}^{\perp}\mathcal{M},\mathcal{M}^{\wedge})$ coincide. 
\end{enumerate}
In addition, assume that $\mathcal{C}$ is a weakly idempotent complete category which has enough injective objects, $\mathrm{pd}\mathcal{M}<\infty$ and $\mathcal{M}$ is a contravariantly finite subcategory of $\mathcal{C}$. 
Then the following condition is also equivalent. 
\begin{enumerate}[\upshape(4)]
\item $\mathcal{M}^{\perp}=\mathcal{M}^{\wedge}$. 
\end{enumerate}
\end{corollary}

\begin{proof}
(1)$\Leftrightarrow$(2): This is clear. 

(2)$\Rightarrow$(3): By Lemma \ref{lem-thickE}, $\mathcal{M}$ is a silting subcategory of $\mathcal{C}$. 
It follows from Proposition \ref{thm-AT22} that $(\mathcal{M}^{\vee},\mathcal{M}^{\wedge})$ is a hereditary cotorsion pair in $\mathcal{C}$. 
Moreover, by Lemma \ref{lem-psilt_wedge}(4) and (5), we have ${}^{\perp}\mathcal{M}=\mathcal{M}^{\vee}$ and $\mathcal{M}^{\perp}=\mathcal{M}^{\wedge}$. 

(3)$\Rightarrow$(1): Since $(\mathcal{M}^{\vee},\mathcal{M}^{\wedge})$ is a hereditary cotorsion pair in $\mathcal{C}$, we have
\begin{align}
\mathcal{C}=\operatorname{\mathsf{cone}}(\mathcal{M}^{\wedge},\mathcal{M}^{\vee})\subseteq \operatorname{\mathsf{cone}}(\mathcal{P}^{\infty},\mathcal{P}^{\infty})\subseteq \mathcal{P}^{\infty}, \notag
\end{align}
and hence the assertion holds. 

(3)$\Rightarrow$(4): This is clear. 

In the following, we assume that $\mathcal{C}$ is a weakly idempotent complete category which has  enough injective objects, $\mathrm{pd}\mathcal{M}<\infty$ and $\mathcal{M}$ is a contravariantly finite subcategory of $\mathcal{C}$. 

(4)$\Rightarrow$(1): Since $\mathrm{pd}\mathcal{M}<\infty$ and $\mathcal{M}$ is a contravariantly finite subcategory of $\mathcal{C}$, it follows from Proposition \ref{prop-Psi}  that    
$(\mathcal{M}^{\vee},\mathcal{M}^{\perp})$ is a hereditary cotorsion pair in $\mathcal{C}$. 
Thus we obtain 
\begin{align}
\mathcal{C}=\operatorname{\mathsf{cone}}(\mathcal{M}^{\perp},\mathcal{M}^{\vee})=\operatorname{\mathsf{cone}}(\mathcal{M}^{\wedge},\mathcal{M}^{\vee})\subseteq \operatorname{\mathsf{cone}}(\mathcal{P}^{\infty},\mathcal{P}^{\infty})\subseteq \mathcal{P}^{\infty}.\notag
\end{align}
Hence the proof is complete. 
\end{proof}

By regarding $\operatorname{\mathsf{mod}}\Lambda$ as an extriangulated category, we have the following result.

\begin{corollary}[{\cite[Theorem 5.1]{CP}}] 
Let $\Lambda$ be a finite dimensional algebra and let $T$ be a tilting $\Lambda$-module. 
Then the following statements are equivalent. 
\begin{enumerate}[\upshape(1)]
\item $\Lambda$ has finite global dimension. 
\item $\operatorname{\mathsf{thick}}T=\operatorname{\mathsf{mod}}\Lambda$. 
\item The pairs $((\operatorname{\mathsf{add}}T)^{\vee},(\operatorname{\mathsf{add}}T)^{\perp})$ and $({}^{\perp}(\operatorname{\mathsf{add}}T),(\operatorname{\mathsf{add}}T)^{\wedge})$ are hereditary cotorsion pairs, and $((\operatorname{\mathsf{add}}T)^{\vee},(\operatorname{\mathsf{add}}T)^{\perp})$ and $({}^{\perp}(\operatorname{\mathsf{add}}T),(\operatorname{\mathsf{add}}T)^{\wedge})$ coincide.
\item $(\operatorname{\mathsf{add}}T)^{\perp}=(\operatorname{\mathsf{add}}T)^{\wedge}$.
\end{enumerate} 
\end{corollary}

\begin{proof}
The global dimension of $\Lambda$ is finite if and only if $\operatorname{\mathsf{mod}}\Lambda=\mathcal{P}^{\infty}$. 
By Corollary \ref{ex-siltilt}, $T$ is a silting object in $\mathcal{P}^{\infty}$.
Since $\operatorname{\mathsf{mod}}\Lambda$ is Hom-finite, $\operatorname{\mathsf{add}}T$ is contravariantly finite.
Thus the assertion follows from Corollary \ref{cor-CP}. 
\end{proof}

We finish this subsection with recovering Auslander--Reiten's result. 
For objects $C_{1}, C_{2}, \ldots, C_{n}$ in $\mathcal{C}$, let $\mathcal{F}(C_{1},C_{2},\ldots, C_{n})$ denote the subcategory consisting of $M\in\mathcal{C}$ which admits a sequence of $\mathfrak{s}$-inflations
\begin{align}
 0=M_{0}\xrightarrow{f_{0}}M_{1}\xrightarrow{f_{1}}M_{2}\rightarrow \cdots \xrightarrow{f_{l-1}}M_{l}=M\notag
\end{align}
such that $\mathrm{Cone}(f_{i})\in\operatorname{\mathsf{add}}C_{j_{i}}$ for each $i\in[0,l-1]$. 
Note that $\mathcal{F}(C_{1},C_{2},\ldots, C_{n})$ is the smallest subcategory of $\mathcal{C}$ which is closed under extensions and contains $C_{1},C_{2},\ldots, C_{n}$.
The following lemma is an extriangulated version of \cite[Proposition 3.8, Corollary 3.10]{AR91}. 

\begin{lemma}\label{lem-finpdim}
Let $\mathcal{C}$ be a Krull--Schmidt extriangulated category which has enough projective objects. 
Assume $\mathcal{C}=\mathcal{F}(C_{1},C_{2},\ldots,C_{n})$ for objects $C_{1}, C_{2}, \ldots, C_{n}$ in $\mathcal{C}$.
Let $\mathcal{X}$ be a contravariantly finite resolving subcategory of $\mathcal{C}$.
Then the following statements hold. 
\begin{enumerate}[\upshape(1)]
\item For each $M\in\mathcal{C}$, there exists an $\mathfrak{s}$-conflation $K\rightarrow Y \rightarrow M\dashrightarrow $ such that $K\in\mathcal{X}^{\perp}$ and $Y\in\mathcal{F}(X_{1}, X_{2}, \ldots, X_{n})$,  where $X_{i}\rightarrow C_{i}$ is a minimal right $\mathcal{X}$-approximation. 
\item $\mathcal{X}=\operatorname{\mathsf{add}}(\mathcal{F}(X_{1}, X_{2}, \ldots, X_{n}))$.
\item $\mathcal{X}\subseteq\mathcal{P}^{\infty}$ if and only if $\mathrm{pd}\mathcal{X}<\infty$. 
\end{enumerate}
\end{lemma}

\begin{proof}
Let $\mathcal{Y}:=\mathcal{F}(X_{1},X_{2}, \ldots, X_{n})$. 

(1) (i) We show that there exists an $\mathfrak{s}$-conflation $K_{i}\rightarrow X_{i} \xrightarrow{\alpha_{i}} C_{i}\dashrightarrow $ such that $K_{i}\in\mathcal{X}^{\perp}$ and $\alpha_{i}:X_{i}\rightarrow C_{i}$ is a minimal right $\mathcal{X}$-approximation for each $i\in[1,n]$. 
By \cite[Lemma 5.14]{AT22}, we have an $\mathfrak{s}$-conflation $K'_{i}\rightarrow X'_{i}\xrightarrow{\beta_{i}} C_{i}\dashrightarrow$ with $K'_{i}\in\mathcal{X}^{\perp}$ and $X'_{i}\in\mathcal{X}$. 
It follows from $K'_{i}\in\mathcal{X}^{\perp}$ that $\beta_{i}$ is a right $\mathcal{X}$-approximation. 
Let $\alpha_{i}$ be a right minimal morphism of $\beta_{i}$. 
Then $\alpha_{i}$ is an $\mathfrak{s}$-deflation and $K_{i}:=\mathrm{Cocone}(\alpha_{i})$ is a direct summand of $K'_{i}\in\mathcal{X}^{\perp}$. 
Thus we obtain a desired $\mathfrak{s}$-conflation. 

(ii)  Let $M\in\mathcal{C}=\mathcal{F}(C_{1},C_{2},\ldots, C_{n})$. 
Then there exists a sequence of $\mathfrak{s}$-inflations
\begin{align}
 0=M_{0}\xrightarrow{f_{0}}M_{1}\xrightarrow{f_{1}}M_{2}\rightarrow \cdots \xrightarrow{f_{l-1}}M_{l}=M\notag
\end{align}
such that $\mathrm{Cone}(f_{i})\in\operatorname{\mathsf{add}}C_{j_{i}}$ for each $i\in[0,l-1]$. 
By induction on $l$, we show that there exists an $\mathfrak{s}$-conflation $K\rightarrow Y \rightarrow M\dashrightarrow $ such that $K\in\mathcal{X}^{\perp}$ and $Y\in\mathcal{Y}$.
If $l=1$, then $M\in\operatorname{\mathsf{add}}C_{j_{0}}$. 
Hence the assertion follows from (i). 
Assume $l\geq 2$. 
Then we have an $\mathfrak{s}$-conflation $M_{l-1}\xrightarrow{f_{l-1}}M\rightarrow \mathrm{Cone}(f_{l-1})\dashrightarrow$ with $\mathrm{Cone}(f_{l-1})\in\operatorname{\mathsf{add}}C_{j_{l-1}}$. 
It follows from (i) that there exists an $\mathfrak{s}$-conflation $K'\rightarrow Y'\rightarrow \mathrm{Cone}(f_{l-1})\dashrightarrow$ such that $K'\in\mathcal{X}^{\perp}$ and $Y'\in\operatorname{\mathsf{add}}X_{j_{l-1}}$. 
By \cite[Proposition 3.15]{NP19}, we have a commutative diagram
\begin{align}
\xymatrix{
&K'\ar@{=}[r]\ar[d]&K'\ar[d]&\\
M_{l-1}\ar[r]\ar@{=}[d]&E\ar[r]\ar[d]&Y'\ar@{-->}[r]^{\delta}\ar[d]&\\
M_{l-1}\ar[r]&M\ar[r]\ar@{-->}[d]&\mathrm{Cone}(f_{l-1})\ar@{-->}[r]\ar@{-->}[d]&\\
&&&.
}\notag
\end{align}
The induction hypothesis implies that there exists an $\mathfrak{s}$-conflations $K_{l-1}\rightarrow Y_{l-1}\xrightarrow{\alpha} M_{l-1}\dashrightarrow$ such that $K_{l-1}\in\mathcal{X}^{\perp}$ and $Y_{l-1}\in\mathcal{Y}$. 
Applying $\mathcal{C}(Y',-)$ to the $\mathfrak{s}$-conflation gives an exact sequence 
\begin{align}
\mathbb{E}(Y',K_{l-1})\rightarrow \mathbb{E}(Y',Y_{l-1})\xrightarrow{\mathbb{E}(Y',\alpha)} \mathbb{E}(Y', M_{l-1})\rightarrow \mathbb{E}^{2}(Y',K_{l-1}).\notag
\end{align}
By $Y'\in\mathcal{X}$ and $K_{l-1}\in\mathcal{X}^{\perp}$, we have $\mathbb{E}(Y',K_{l-1})=0=\mathbb{E}^{2}(Y',K_{l-1})$.
Thus there exists $\eta\in\mathbb{E}(Y',Y_{l-1})$ such that $\delta=\mathbb{E}(Y',\alpha)(\eta)=\alpha^{\ast}\eta$. 
Hence we obtain a commutative diagram
\begin{align}
\xymatrix{
K_{l-1}\ar@{=}[r]\ar[d]&K_{l-1}\ar[d]&&\\
Y_{l-1}\ar[r]\ar[d]_{\alpha}&Y\ar[r]\ar[d]&Y'\ar@{-->}[r]^{\eta}\ar@{=}[d]&\\
M_{l-1}\ar[r]\ar@{-->}[d]&E\ar[r]\ar@{-->}[d]&Y'\ar@{-->}[r]^{\delta}&\\
&&&.
}\notag
\end{align}
Since $\mathcal{Y}$ is closed under extensions, we have $Y\in\mathcal{Y}$. 
By (ET4)$^{\mathrm{op}}$, we have a commutative diagram
\begin{align}
\xymatrix{
K_{l-1}\ar@{=}[r]\ar[d]&K_{l-1}\ar[d]&&\\
K\ar[r]\ar[d]&Y\ar[r]\ar[d]&M\ar@{-->}[r]\ar@{=}[d]&\\
K'\ar[r]\ar@{-->}[d]&E\ar[r]\ar@{-->}[d]&M\ar@{-->}[r]&\\
&&&.
}\notag
\end{align}
Since $\mathcal{X}^{\perp}$ is closed under extensions, we have $K\in\mathcal{X}^{\perp}$, and hence  $K\rightarrow Y \rightarrow M\dashrightarrow$ is a desired $\mathfrak{s}$-conflation. 

(2) Let $X\in \mathcal{X}$. By (1), there exists an $\mathfrak{s}$-conflation $K\rightarrow Y\rightarrow X\dashrightarrow$ such that $K\in\mathcal{X}^{\perp}$ and $Y\in \mathcal{Y}$. 
By $X\in \mathcal{X}$, the $\mathfrak{s}$-conflation splits.
Hence $X\in \operatorname{\mathsf{add}}\mathcal{Y}$.
Since the converse inclusion is clear, we have the assertion.

(3) Since the ``if'' part is clear, we show the ``only if'' part.
By $X_{i}\in\mathcal{X}\subseteq \mathcal{P}^{\infty}$, there exists $d\geq 0$ such that, for each $i\in[1,n]$, $\operatorname{\mathsf{add}}X_{i}\subseteq(\operatorname{\mathsf{proj}}\mathcal{C})^{\wedge}_{d}$ holds. By Lemma \ref{lem-psilt_wedge}(1), the subcategory $(\operatorname{\mathsf{proj}}\mathcal{C})^{\wedge}_{d}$ is closed under extensions and direct summands.
Therefore it follows from (2) that $\mathcal{X}=\operatorname{\mathsf{add}}(\mathcal{F}(X_{1},X_{2},\ldots, X_{n}))\subseteq (\operatorname{\mathsf{proj}}\mathcal{C})^{\wedge}_{d}$.
Namely, $\mathrm{pd}\mathcal{X}\leq d$.
\end{proof}

Now, we are ready to recover Auslander--Reiten's result. 

\begin{corollary}[{\cite[Theorem 5.5]{AR91}}]\label{cor_AR}
Let $\Lambda$ be a finite dimensional algebra. 
Then $T \mapsto (\operatorname{\mathsf{add}}T)^{\vee}$ gives a bijection between the set of isomorphism classes of basic tilting modules and the set of contravariantly finite resolving subcategories consisting of modules with finite projective dimension.
\end{corollary}

\begin{proof}
By Proposition \ref{prop-siltpfin}(1), $\Lambda$ is a silting object in $\mathcal{P}^{\infty}$.
Thus it follows from \cite[Proposition 5.4]{AT22} that each silting subcategory of $\mathcal{P}^{\infty}$ has an additive generator.
Thus we obtain $\operatorname{\mathsf{silt}}\mathcal{P}^{\infty}=\{ \mathcal{M}\in \operatorname{\mathsf{silt}}\mathcal{P}^{\infty}\mid \mathrm{pd}\mathcal{M}<\infty\}$.
Since $\operatorname{\mathsf{mod}}\Lambda$ is Hom-finite, all silting subcategories of $\mathcal{P}^{\infty}$ are contravariantly finite.
By Theorem \ref{thm-AR2} and Lemma \ref{lem-finpdim}(3), we have mutually inverse bijections
\begin{align}
\xymatrix{
\operatorname{\mathsf{silt}}\mathcal{P}^{\infty}\ar@<0.5ex>[d]^-{\varphi} \\
\{\mathcal{X}:\textnormal{contravariantly finite resolving subcategory of }\mathcal{C}\mid \mathcal{X} \subseteq\mathcal{P}^{\infty}\}\ar@<0.5ex>[u]^{\psi}. 
}\notag\end{align}
Thus the assertion follows from Corollary \ref{ex-siltilt}.
\end{proof}

\section{Examples of silting objects in an extriangulated category}

In this section, we give an example of silting objects and silting mutation. 
Let $\mathcal{C}$ be an extriangulated category.
Assume that $R$ is a field and $\mathcal{C}$ is $\mathbb{E}$-finite, that is, $\mathbb{E}(X,Y)$ is finite dimensional for all $X,Y\in\mathcal{C}$.
Recall the notion of a standardizable set in $\mathcal{C}$.

\begin{definition}
Let $\Theta:=(\Theta(1), \Theta(2), \ldots, \Theta(n))$ be an ordered set of objects in $\mathcal{C}$.
We call $\Theta$ a \emph{standardizable set} in $\mathcal{C}$ if it satisfies the following conditions.
\begin{itemize}
\item[(S1)] For each $i\in [1,n]$, $\Theta(i)$ is indecomposable.
\item[(S2)] If $i>j$, then $\mathcal{C}(\Theta(i), \Theta(j))=0$ holds.
\item[(S3)] If $i\geq j$, then $\mathbb{E}(\Theta(i), \Theta(j))=0$ holds.
\end{itemize}
\end{definition}

Let $\mathcal{F}(\Theta):=\mathcal{F}(\Theta(1),\Theta(2),\ldots,\Theta(n))$, $\mathcal{F}(\Theta(\geq i)):=\mathcal{F}(\Theta(i),\Theta(i+1),\ldots,\Theta(n))$ and $\mathcal{F}(\Theta(\leq i)):=\mathcal{F}(\Theta(1),\Theta(2),\ldots,\Theta(i))$. 
Since $\mathcal{F}(\Theta)$ is closed under extensions, it becomes an extriangulated category.   
Note that $\mathcal{F}(\Theta)=\operatorname{\mathsf{add}}\Theta(n)\ast\operatorname{\mathsf{add}}\Theta(n-1)\ast\cdots\ast\operatorname{\mathsf{add}}\Theta(1)$ (e.g., \cite[Proposition 3.9(1)]{AT}).

Following Dlab--Ringel's standardization method (\cite{DR}), we construct a projective generator and an injective cogenerator of $\mathcal{F}(\Theta)$.
An object $X \in \mathcal{F}(\Theta)$ is called a \emph{projective generator} if it is projective in $\mathcal{F}(\Theta)$ and $\mathcal{F}(\Theta)\subseteq \operatorname{\mathsf{cone}}(\mathcal{F}(\Theta),\operatorname{\mathsf{add}}X)$ holds.
Dually, we define an \emph{injective cogenerator} of $\mathcal{F}(\Theta)$.

\begin{proposition}[{\cite[Proposition 3.10(1)]{AT}}]\label{prop-std}
Let $\Theta:=(\Theta(1), \Theta(2), \ldots,\Theta(n))$ be a standardizable set in $\mathcal{C}$. 
Then the following statements hold.
\begin{enumerate}[\upshape(1)]
\item For each $i\in[1,n]$, there exists an $\mathfrak{s}$-conflation 
\begin{align}
K(i) \rightarrow P(i)\rightarrow \Theta(i) \dashrightarrow \notag 
\end{align}
such that $K(i)\in\mathcal{F}(\Theta(\geq i+1))$ and $P(i)$ is a projective object in $\mathcal{F}(\Theta)$. 
In particular, $P_{\Theta}:=\oplus_{i=1}^{n}P(i)$ is a projective generator of $\mathcal{F}(\Theta)$. 
\item For each $i\in[1,n]$, there exists an $\mathfrak{s}$-conflation 
\begin{align}
\Theta(i) \rightarrow I(i)\rightarrow C(i) \dashrightarrow \notag 
\end{align}
such that $C(i)\in\mathcal{F}(\Theta(\leq i-1))$ and $I(i)$ is an injective object in $\mathcal{F}(\Theta)$. 
In particular, $I_{\Theta}:=\oplus_{i=1}^{n}I(i)$ is an injective cogenerator of $\mathcal{F}(\Theta)$. 
\end{enumerate}
\end{proposition}

Now, we give an example of a silting object of $\mathcal{F}(\Theta)$.

\begin{proposition}\label{prop-stdsilt}
If $\Theta:=(\Theta(1), \Theta(2), \ldots,\Theta(n))$ is a standardizable set in $\mathcal{C}$, then $P_{\Theta}$ and $I_{\Theta}$ are silting objects of $\mathcal{F}(\Theta)$. In particular, $\operatorname{\mathsf{silt}}\mathcal{F}(\Theta)$ has a greatest element $P_{\Theta}$ and a least element $I_{\Theta}$.
\end{proposition}

\begin{proof}
By Proposition \ref{prop-std}(1), $P_{\Theta}$ is a projective object in $\mathcal{F}(\Theta)$, and hence it is presilting in $\mathcal{F}(\Theta)$. 
Let $\operatorname{\mathsf{thick}}_{\mathcal{F}(\Theta)}P_{\Theta}$ denote the smallest thick subcategory of $\mathcal{F}(\Theta)$ containing $P_{\Theta}$. 
Note that $\operatorname{\mathsf{thick}}_{\mathcal{F}(\Theta)}P_{\Theta}$ is a subcategory of $\mathcal{C}$ which is closed under extensions. 
We show that $\Theta(i)\in\operatorname{\mathsf{thick}}_{\mathcal{F}(\Theta)}P_{\Theta}$ by induction on $i$.  
If $i=n$, then $\Theta(n)\cong P(n)\in\operatorname{\mathsf{thick}}_{\mathcal{F}(\Theta)}P_{\Theta}$. 
Assume $i \leq n-1$. 
By Proposition \ref{prop-std}(1), there exists an $\mathfrak{s}$-conflation $K(i)\rightarrow P(i) \rightarrow \Theta(i)\dashrightarrow$ such that $K(i)\in\mathcal{F}(\Theta(\geq i+1))$ and $P(i)$ is a projective object in $\mathcal{F}(\Theta)$. 
The induction hypothesis implies that $\Theta(j)\in\operatorname{\mathsf{thick}}_{\mathcal{F}(\Theta)}P_{\Theta}$ for each $j\geq i+1$. 
Since $\operatorname{\mathsf{thick}}_{\mathcal{F}(\Theta)}P_{\Theta}$ is closed under extensions, we have $K(i)\in\operatorname{\mathsf{thick}}_{\mathcal{F}(\Theta)}P_{\Theta}$. 
Hence the assertion follows from the fact that $\operatorname{\mathsf{thick}}_{\mathcal{F}(\Theta)}P_{\Theta}$ is closed under cones. 
Since $\mathcal{F}(\Theta)$ is the smallest subcategory which is closed under extensions and contains $\Theta(1), \ldots, \Theta(n)$, we obtain $\mathcal{F}(\Theta)\subseteq\operatorname{\mathsf{thick}}_{\mathcal{F}(\Theta)}P_{\Theta}$.
Hence $P_{\Theta}$ is a silting object in $\mathcal{F}(\Theta)$. 
Dually, $I_{\Theta}$ is a silting object in $\mathcal{F}(\Theta)$. 
\end{proof}

If $\mathcal{F}(\Theta)$ is finite-type (i.e., there are only finitely many non-isomorphism indecomposable objects in $\mathcal{F}(\Theta)$), then all silting objects in $\mathcal{F}(\Theta)$ are reachable by iterated mutation. 

\begin{corollary}
If $\mathcal{F}(\Theta)$ is finite-type, then each silting object in $\mathcal{F}(\Theta)$ is obtained from $P_{\Theta}$ by iterated irreducible left mutation.
\end{corollary}

\begin{proof}
By Proposition \ref{prop-stdsilt}, we obtain $\operatorname{\mathsf{silt}}\mathcal{F}(\Theta)=\{ M\in\operatorname{\mathsf{silt}}\mathcal{F}(\Theta)\mid P_{\Theta}\geq M\geq I_{\Theta}\}$.
Since $\mathcal{F}(\Theta)$ is finite-type, the set $\operatorname{\mathsf{silt}}\mathcal{F}(\Theta)$ is finite.
Thus the assertion follows from Corollary \ref{cor-finsilt}.
\end{proof}

Using standardizable sets, we give a concrete example of silting objects and silting mutation.  

\begin{example}
Assume that $R$ is an algebraically closed field and $\Lambda:=R(1\to 2\to 3\to4)$ be the path algebra.
Consider the bounded derived category $\mathcal{D}$ of $\operatorname{\mathsf{mod}}\Lambda$. 
Then the Auslander--Reiten quiver of $\mathcal{D}$ is as follows.
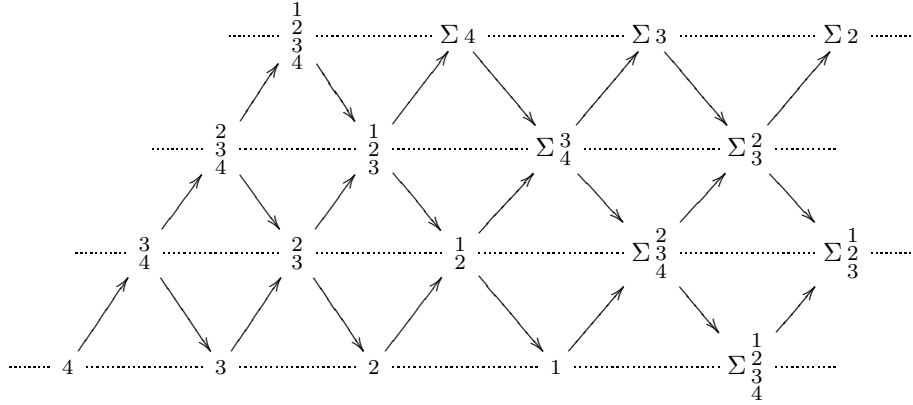
\begin{align}
\xymatrix @R=5mm @C=5mm{
&&&&{\begin{smallmatrix}1\\2\\3\\4\end{smallmatrix}}\ar[rd]\ar@{.}[l]&&
{\textnormal{\small{$\Sigma$}}\begin{smallmatrix}4\end{smallmatrix}}\ar[rd]\ar@{.}[ll]&&
{\textnormal{\small{$\Sigma$}}\begin{smallmatrix}3\end{smallmatrix}}\ar[rd]\ar@{.}[ll]&&{\textnormal{\small{$\Sigma$}}\begin{smallmatrix}2\end{smallmatrix}}\ar@{.}[ll]&\ar@{.}[l]\\
&&&{\begin{smallmatrix}2\\3\\4\end{smallmatrix}}\ar[ru]\ar[rd]\ar@{.}[l]&&
{\begin{smallmatrix}1\\2\\3\end{smallmatrix}}\ar[ru]\ar[rd]\ar@{.}[ll]&&
{\textnormal{\small{$\Sigma$}}\begin{smallmatrix}3\\4\end{smallmatrix}}\ar[ru]\ar[rd]\ar@{.}[ll]&&{\textnormal{\small{$\Sigma$}}\begin{smallmatrix}2\\3\end{smallmatrix}}\ar[ru]\ar[rd]\ar@{.}[ll]&\ar@{.}[l]\\
&&{\begin{smallmatrix}3\\4\end{smallmatrix}}\ar[ru]\ar[rd]\ar@{.}[l]&&
{\begin{smallmatrix}2\\3\end{smallmatrix}}\ar[ru]\ar[rd]\ar@{.}[ll]&&
{\begin{smallmatrix}1\\2\end{smallmatrix}}\ar[ru]\ar[rd]\ar@{.}[ll]&&{\textnormal{\small{$\Sigma$}}\begin{smallmatrix}2\\3\\4\end{smallmatrix}}\ar[ru]\ar[rd]\ar@{.}[ll]&&{\textnormal{\small{$\Sigma$}}\begin{smallmatrix}1\\2\\3\end{smallmatrix}}\ar@{.}[ll]&\ar@{.}[l]\\
&{\begin{smallmatrix}4\end{smallmatrix}}\ar[ru]\ar@{.}[l]&&\textnormal{$\begin{smallmatrix}3\end{smallmatrix}$}\ar[ru]\ar@{.}[ll]&&
{\begin{smallmatrix}2\end{smallmatrix}}\ar[ru]\ar@{.}[ll]&&\textnormal{$\begin{smallmatrix}1\end{smallmatrix}$}\ar[ru]\ar@{.}[ll]&&\textnormal{$\textnormal{\small{$\Sigma$}}\begin{smallmatrix}1\\2 \\3 \\4 \end{smallmatrix}$}\ar[ru]\ar@{.}[ll]&\ar@{.}[l]
}\notag
\end{align}
Let $\Theta(1):=\textnormal{\small{$\Sigma$}}\begin{smallmatrix}1\\2\\3\end{smallmatrix}$, $\Theta(2):=\begin{smallmatrix}1\end{smallmatrix}$, $\Theta(3):=\begin{smallmatrix}2\\ 3\end{smallmatrix}$, $\Theta(4):=\begin{smallmatrix}4\end{smallmatrix}$. 
Note that the extriangulated category $\mathcal{F}(\Theta)$ is neither a triangulated category nor an exact category. 
Since $(\Theta(1),\Theta(2),\Theta(3),\Theta(4))$ is a standardizable set, $P_{\Theta}=P(1)\oplus P(2)\oplus P(3)\oplus P(4)$ is a silting object in $\mathcal{F}(\Theta)$, where $P(1)=0$, $P(2)=\begin{smallmatrix}1\\2\\3\\4\end{smallmatrix}$, $P(3)=\begin{smallmatrix}2\\3\\4\end{smallmatrix}$, $P(4)=\Theta(4)$.
By left mutation, we have a finite connected component 
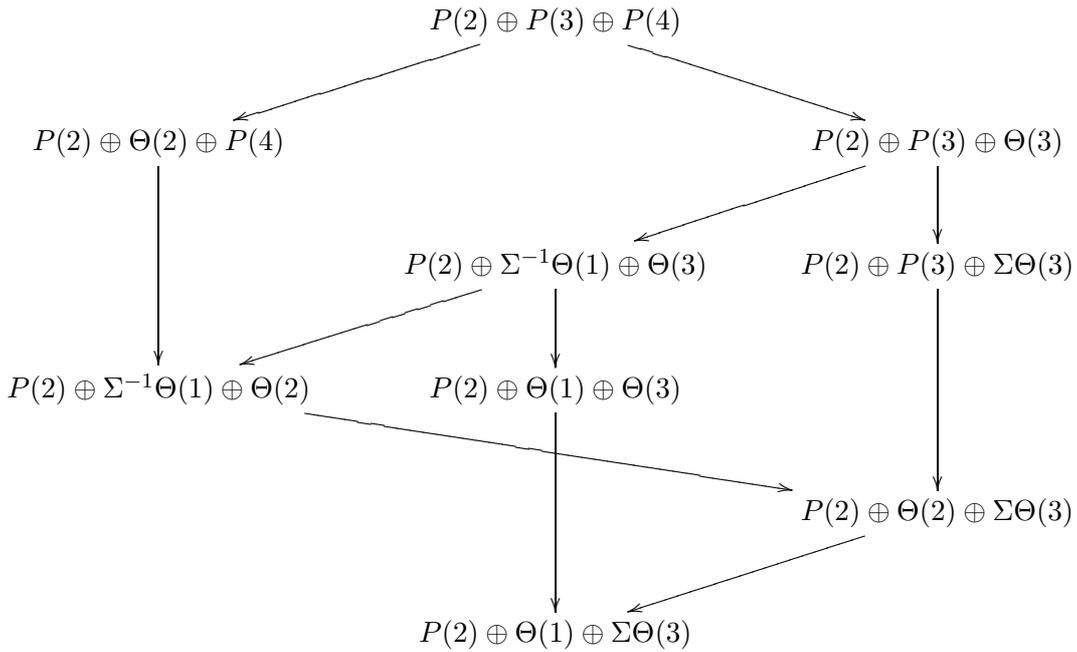
\begin{align}
\xymatrix @R=10mm @C=10mm{
&P(2)\oplus P(3)\oplus P(4)\ar[ld]\ar[rd]\\
P(2)\oplus \Theta(2)\oplus P(4)\ar[dd]&&P(2)\oplus P(3)\oplus \Theta(3)\ar[ld]\ar[d]\\
&P(2)\oplus \Sigma^{-1}\Theta(1)\oplus \Theta(3)\ar[ld]\ar[d]&P(2)\oplus P(3)\oplus \Sigma\Theta(3)\ar[dd]\\
P(2)\oplus \Sigma^{-1}\Theta(1)\oplus \Theta(2)\ar[rrd]&P(2)\oplus \Theta(1)\oplus \Theta(3)\ar[dd]\\
&&P(2)\oplus \Theta(2)\oplus \Sigma \Theta(3)\ar[ld]\\
&P(2)\oplus \Theta(1)\oplus \Sigma \Theta(3)}\notag
\end{align}
By Corollary \ref{cor-mutend}, this coincides with $Q(\operatorname{\mathsf{silt}}\mathcal{F}(\Theta))$.
\end{example}

\subsection*{Acknowledgements}
The first author is supported by JSPS KAKENHI Grant Number JP20K14291.
The second author is supported by JSPS KAKENHI Grant Number JP19K14513.
The second author would like to thank Takuma Aihara for the comments about Remark \ref{remark-basic1845}(1).

\end{document}